\documentclass[10pt,twoside,reqno,a4paper]{article}
\usepackage{mathptmx}
\usepackage{amsmath}
\usepackage{amsfonts}
\usepackage{amsthm}
\usepackage{amssymb} 
\usepackage{array}
\usepackage{geometry}
\usepackage[bookmarks=true,colorlinks=true, pdfstartview=FitV, linkcolor=black, citecolor=blue, urlcolor=black]{hyperref}
\usepackage{movie15}

\usepackage{graphicx}
\usepackage{overpic}
\usepackage{a4wide}

\usepackage{color}
\definecolor{DarkRed}{rgb}{0.55,.00,0.2}
\definecolor{DarkBlue}{rgb}{0.25,.00,0.75}
\definecolor{DarkGrey}{rgb}{0.15,.15,0.15}
\definecolor{LightGrey}{rgb}{0.55,.55,0.55}

\usepackage{comment}

\numberwithin{equation}{section}

\newtheorem{theorem}{Theorem}[section]
\newtheorem{proposition}[theorem]{Proposition}
\newtheorem{lemma}[theorem]{Lemma}
\newtheorem{corollary}[theorem]{Corollary}

\newtheorem{Remark}[theorem]{Remark}
\newenvironment{remark}{\begin{Remark} \sf}{\end{Remark}}
\newtheorem{example}{Example}

\newcommand{\ds}{\displaystyle}
\newcommand{\e}{{\rm e}}

\newcommand{\mbf}[1]{\mathbf{#1}}

\begin{document}

\title{\bf The Kontorovich-Lebedev transform as a map between $d$-orthogonal polynomials}

\maketitle
\begin{center}
\begin{tabular}{c@{\hspace*{2cm}}c}
{\bf Ana F. Loureiro} & {\bf S. Yakubovich}\\ 
anafsl@fc.up.pt  & syakubov@fc.up.pt \\
Centro de Matem\'atica  & Department of Mathematics, \\ 
Fac. Sciences of University of Porto & Fac. Sciences of University of Porto,\\
Rua do Campo Alegre,  687 & Rua do Campo Alegre,  687\\ 
4169-007 Porto (Portugal)&4169-007 Porto (Portugal)
\end{tabular}
\end{center}

\begin{abstract} A slight modification of the Kontorovich-Lebedev transform is an automorphism on the vector space of polynomials. The action of this $KL_{\alpha}$-transform over certain polynomial sequences will be under discussion, and a special attention will be given the  d-orthogonal ones. For instance, the Continuous Dual Hahn polynomials appear as the $KL_{\alpha}$-transform of a $2$-orthogonal sequence of Laguerre type. Finally, all the orthogonal polynomial sequences whose $KL_{\alpha}$-transform is a $d$-orthogonal sequence will be characterized: they are essencially semiclassical polynomials fulfilling particular conditions and $d$ is even. The Hermite and Laguerre polynomials are the classical solutions to this problem. 
\end{abstract}




\section{Introduction }

Throughout  the text, $\mathbb{N}$ will denote the set of all positive integers, $\mathbb{N}_{0}=\mathbb{N}\cup \{0\}$, whereas $\mathbb{R}$ and $\mathbb{C}$  the field of the real and complex numbers, respectively. The notation $\mathbb{R}_{+}$   corresponds to the set of all positive real numbers. The present investigation is primarily targeted at analysis of sequences of polynomials whose degrees equal its order, which will be shortly called as PS. Whenever the leading coefficient of each of its polynomials equals $1$, the PS is said to be a MPS ({\it monic polynomial sequence}). A PS or a MPS forms a basis of the vector space of polynomials with coefficients in $\mathbb{C}$, here denoted as $\mathcal{P}$. The conventions $\prod_{\sigma=0}^{-1}:=1$, $\sum_{\nu=k}^{m}=0$, for any integers $k,m$ such that $m\leqslant k-1$, are assumed for the whole of the text. Further notations are introduced as needed. 

A very general function included most of the known special functions as particular cases is the Meijer G-function 
$$
	\mathop{G_{p,q}^{m,n}}\left(z;{\mbf{a} \atop \mbf{b}}\right)
	:= \mathop{G_{p,q}^{m,n}}\left(z;{a_{1},\ldots,a_{p} \atop b_{1},\ldots,b_{q}}\right)
$$
defined via the Mellin-Barnes integral (the reciprocal formula of the Mellin transform) 
\begin{equation}\label{Meijer G}
	\mathop{G_{p,q}^{m,n}}\left(z;{\mbf{a} \atop \mbf{b}}\right)
	:=\mathop{G_{p,q}^{m,n}}\left(z;{a_{1},\ldots,a_{p} \atop b_{1},\ldots,b_{q}}\right)
	=\frac{1}{2\pi i}\int _{{L}}
		{\frac{\prod\limits _{{l =1}}^{m}\mathop{\Gamma\/}\nolimits\!\left(b_{{l }}-s\right)
			\prod\limits _{{l =1}}^{n}\mathop{\Gamma\/}\nolimits\!\left(1-a_{{l }}+s\right)}
		{\prod\limits _{{l =m}}^{{q-1}}\mathop{\Gamma\/}\nolimits\!\left(1-b_{{l +1}}+s\right)
		\prod\limits _{{l =n}}^{{p-1}}\mathop{\Gamma\/}\nolimits\!\left(a_{{l +1}}-s\right) }}z^{s}ds,
\end{equation}
as long as $0 \leqslant m \leqslant q$ and $0 \leqslant n \leqslant p$, where $m, n, p$ and $q$ are integer numbers, $a_{k}-b_{i} \neq 1, 2, 3,\ldots $ for $k = 1, 2, ..., n$ and $j = 1, 2, ..., m$, which implies that no pole of any $\Gamma(b_{j} - s), j = 1, 2, ..., m,$  coincides with any pole of any $\Gamma(1 - a_{k} + s), k = 1, 2, ..., n$ and $z \neq 0$. 

Such function includes as special cases all the generalized hypergeometric functions ${}_{p}F_{q}$, Mathieu functions, among others. 

In 1964, Wimp formally introduced the general index transform over parameters of the Meijer $G$-function 
\begin{equation}
	F(\tau) = \int_{0}^{\infty} 
	\mathop{G_{p+2,q}^{m,n+2}}\left({x; {1-\mu+i\tau,1-\mu-i\tau, (a_{p}) \atop (b_{q})}}\right) f(x) dx \ , 
	\label{Wimp Direct}
\end{equation}
whose inversion formula was established in 1985 \cite{Yak1985} by the second author (see \cite{Springer,YakuLuchko,YakubovichBook1996})
\begin{equation}\label{Wimp Inverse}
	f(x) = \frac{1}{\pi^{2}} \int_{0}^{\infty} \tau \sinh(2\pi\tau) F(\tau) 
	\mathop{G_{p+2,q}^{q-m,p-n+2}}\left({x; {\mu+i\tau,\mu-i\tau, -(a_{p}^{n+1}) ,-(a_{n}) 
		\atop - ( b_{q}^{m+1}) , - (b_{m})}}\right) d\tau \ . 
\end{equation}

The Kontorovich-Lebedev (KL)-transform comes out as the very simple case of this general transform, arising upon the choice of parameters $m=n=p=q=0$. The reciprocal pair of transformations \eqref{Wimp Direct}-\eqref{Wimp Inverse} incorporates all the existent index transforms in the literature such as Mehler-Fock, Olevski-Fourier-Jacobi, Whittaker, Lebedev's transform with a combination of modified Bessel functions, among others. We notice that all these index transforms can be obtained through the composition of the KL- with the Mellin type convolution transforms. For further reading on the subject we refer to \cite{Springer,YakubovichBook1996,YakuLuchko}. 

The passage of a polynomial sequence to another performed by an integral transform was already considered, either by mean of the convolution type integral transform like the Fourier or the Mellin \cite{Atakishiyev,Koelink} or by the index integral transforms of the Fourier-Jacobi-Olevski \cite{Groenevelt Wilson,Koornwinder85,Koornwinder88}, among others (the reference list has no pretensions of completeness). Therein we may read the application of integral transforms to deal with precise examples of orthogonal sequences. The present work aims to give some basic notions about the action of the canonical index integral transforms while acting over the so-called $d$-orthogonal sequences, which are essentially sets of polynomials fulfilling a recurrence relation of order $d+1$. 

Our focus of attention will be on a slight modification of the KL-transform -- the $KL_{\alpha}$ -- which has the merit of linking certain families of $d$-orthogonal polynomials in a natural way. After reviewing in \S\ref{Sec: Preliminary} all the necessary properties for the sequel,  in  \S\ref{Sec: KL alpha} the key  properties of this operator will be given. Among them, the connection between functional relations of the $KL_{\alpha}$ image of certain differential relations satisfied by polynomial sequences. Thus, on \S\ref{Sec: Hypergeom polys} we provide enlightening examples of $d$-orthogonal polynomial sequences mapped by the $KL_{\alpha}$-operator into another $\widetilde{d}$-orthogonal polynomial sequences. It is the case, for instance, of those that are Appell sequences (see \S\ref{subsec: Appell}) or those whose reversed polynomial sequence possesses the Appell property (like the orthogonal Laguerre polynomials). Precisely, on \S\ref{subsect: Rev Appell} we will show that the Continuous Dual Hahn orthogonal polynomials are the $KL_{\alpha}$-image of the $2$-orthogonal sequence connected to the Bateman function and treated in \cite{BCheikhDouak2000}. Finally, on \S\ref{sec: KL MOPS to dMOPS} we seek all the orthogonal polynomial sequences that are mapped by the $KL_{\alpha}$-operator into $d$-orthogonal polynomial sequences. These are necessarily semiclassical polynomials of a certain type whose class ranges between $\max(d/2-2,0)$ and $d/2$, where $d$ must be an even integer - see Theorem \ref{Thm: classification}. Among the classical polynomials, only the Hermite and Laguerre polynomials share this property, while, among the semiclassical examples we quote the generalized Hermite polynomials.

\section{Preliminary results}\label{Sec: Preliminary}

In a previous work \cite{LouYak2012} we have shown that the modified  {\it KL-transform} defined by (see \cite{Lebedev,Sneddon} \cite[Th. 6.3]{YakuLuchko})
\begin{equation} \label{KL s directa}
	KL[f](\tau) =\frac{4 \sinh (\pi\tau/2)}{\pi\tau}
		\int_{0}^{\infty} K_{i\tau}(2\sqrt{x}) f(x) dx \ , 
\end{equation}
where $K_{\nu}(z)$ represents the modified Bessel function (also known as Macdonald function) \cite[Vol.II]{Bateman} is an automorphism essentially mapping monomials into central factorials.  Precisely, for positive real values of $\tau $ we have (see the relation (2.16.2.2) in \cite{PrudnikovMarichev}), 
\begin{align}
	& \label{KL s xn}
	\begin{array}{lcl}
	KL[x^n](\tau)& =& \ds  
	\frac { 4\sinh(\pi  {\tau}/2)}{\pi {\tau}} \int_{0}^\infty K_{i {\tau}}(2\sqrt{x}) x^n dx 
	= \prod_{\sigma=1}^n \left(\sigma^2 + \left(\frac{\tau}{2}\right)^2\right) \\
	&=& \ds \left(1- i \tau/2\right)_{n}\left(1+ i {\tau/2}\right)_{n}
	\ ,\  n\in\mathbb{N}_{0},
	\end{array}
\end{align}
where the $(x)_{n}$ represents the {\it Pochammer symbol}: $(x)_{n}:=\prod\limits_{\sigma=0}^{n-1}(x+\sigma)$ when $n\geqslant 1$ and $(x)_{0}=1$. As $\tau \to 0$, \eqref{KL s xn} becomes 
\begin{equation}\label{MomentsK0}
	KL[x^n](0)=   (n!)^{2}\ , \  n\in\mathbb{N}_{0}.
\end{equation}

The KL-transform \eqref{KL s directa} has a reciprocal inversion formula, namely  
\begin{equation}\label{KL s inverse}
	x\ f(x)= \frac{1}{2\pi} \lim_{\lambda\to \pi -} 
		\int_{0}^{\infty} \tau^{2} \cosh(\lambda\tau/2) K_{i\tau}(2\sqrt{x})
		\ KL[f](\tau)d\tau \ .
\end{equation}

The formula \eqref{KL s directa} is valid for any continuous function  $f \in L_{1}\left(\mathbb{R}_{+}, K_{0}(2\mu\sqrt{x}) dx \right)$, $0<\mu<1$, in a neighborhood of each $x\in\mathbb{R}_{+}$ where $f(x)$ has bounded variation.
The kernel of such transformation is the modified Bessel function (also called {\it MacDonald function}) $K_{2i{\tau}}(2\sqrt{x})$ of purely imaginary index, which is real valued and can be defined by integrals of Fourier type 
\begin{align}\label{Cosine Fourier K}
	& K_{i{\tau} }(2\sqrt{x}) = \int_{0}^\infty \e^{-2\sqrt{x} \cosh(u)}\cos({\tau} \, u)du \ , \ x\in\mathbb{R}_{+}, \ \tau\in\mathbb{R}_{+} .
\end{align}
and it is an eigenfunction of the operator 
\begin{equation}\label{op A }
	\mathcal{A} \ = \ x^{2} \frac{d^{2}}{dx^{2}} + x \frac{d}{dx} -x  \ 
	= \  x \frac{d}{dx} x \frac{d}{dx} -x
\end{equation} 
insofar as 
\begin{equation}\label{A Kitau}
	\mathcal{A} K_{ i {\tau} }(2\sqrt{x}) = -\left(\frac{\tau}{2}\right)^{2} \ K_{i{\tau}}(2\sqrt{x}) \ .
\end{equation}

Besides, $K_{\nu}(2\sqrt{x})$ reveals the asymptotic behaviour with respect to $x$ \cite[Vol. II]{Bateman}\cite{YakubovichBook1996}
\begin{align}
	& K_{\nu}(2\sqrt{x}) = \frac{\sqrt{\pi}}{2 \ x^{1/4}}  \e^{-2\sqrt{x}}[1+O(1/\sqrt{x})] , \quad x\rightarrow +\infty,
		\label{Knu at infty}\\
	& K_{\nu}(2\sqrt{x}) =O(x^{-{\rm Re}(\nu)/2}) \ , \ \text{if Re}(\nu)\neq0, \quad 
	\text{ and } \quad K_{0}(2\sqrt{x}) =O(\log x)  \ , \quad x\rightarrow 0. \label{K0 at 0}
\end{align}
and it is valid the  following inequality 
\begin{equation}\label{ineq Kitau} 
	\left| \frac{\partial^m K_{i\tau}(x)}{\partial x^m } \right| \leqslant   {\rm e}^{-\delta \tau} K_{m}(x\cos \delta) , \quad x>0, \ \tau >0, \ m \in\mathbb{N}_{0}
\end{equation} with $\delta \in (0,\pi/2)$. 

As orthogonal polynomials are at the center of discussion, we hereby recall some of foremost important properties to the understanding of the achieved results. 

The dual sequence $\{v_{n}\}_{n\geqslant 0}$ of a given MPS $\{Q_{n}\}_{n\geqslant 0}$ belong to the dual space $\mathcal{P}'$ of $\mathcal{P}$ and whose elements are uniquely defined by 
$$
	\langle v_{n},Q_{k}  \rangle := \delta_{n,k}, \; n,k\geqslant 0,
$$ 
where $\delta_{n,k}$ represents the {\it Kronecker delta} function. Its first element, $u_{0}$, earns the special name of  {\it canonical form} of the MPS. Here, by $\langle u,f\rangle$ we mean the action of $u\in\mathcal{P}'$ over $f\in\mathcal{P}$, but a special notation is given to the action over the elements  of the canonical sequence $\{x^{n}\}_{n\geqslant 0}$ -- the {\it moments of $u\in\mathcal{P}'$}:   $(u)_{n}:=\langle u,x^{n}\rangle, n\geqslant 0 $. 
Any element $u$ of $\mathcal{P}'$ can be written in a series of any dual sequence $\{ \mathbf{v}_{n}\}_{n\geqslant 0}$ of a MPS  $\{Q_{n}\}_{n\geqslant 0}$ \cite{MaroniTheorieAlg}: 
\begin{equation} \label{u in terms of un}
	u = \sum_{n\geqslant 0} \langle u , Q_{n} \rangle \;{v}_{n} \; .
\end{equation}
Differential equations or other kind of linear relations realized by the elements of the dual sequence can be deduced by transposition of those relations fulfilled by the elements of the corresponding MPS, insofar as a linear operator $T:\mathcal{P}\rightarrow\mathcal{P}$  
has a transpose $^{t}T:\mathcal{P}'\rightarrow\mathcal{P}'$ defined by 
\begin{equation}\label{Ttranspose}
	\langle{}^{t}T(u),f\rangle=\langle u,T(f)\rangle\,,\quad u\in\mathcal{P}',\: f\in\mathcal{P}.
\end{equation}
For example, for any form $u$ and any polynomial $g$, let $ Du=u'$ and $gu$ be the forms defined as usual by
$
	\langle u',f\rangle :=-\langle u , f' \rangle \ ,\  \langle gu,f\rangle :=\langle u, gf\rangle ,
$  
where $D$ is the differential operator \cite{MaroniTheorieAlg}. Thus, $D$ on forms is minus the transpose of the differential operator $D$ on polynomials.

Whenever there is a form $v\in\mathcal{P}'$ such that $\langle v , Q_{n} Q_{m} \rangle = k_{n} \delta_{n,m}$ with $k_{n}\neq0$ for all $n,m\in\mathbb{N}_{0}$ \cite{MaroniTheorieAlg,MaroniVariations} for some sequence $\{Q_{n}\}_{n\geqslant0}$, then $v$ is called a {\it regular} form. The PS $\{Q_{n}\}_{n\geqslant 0}$ is then said to be orthogonal with respect to $v$ and we can assume the system (of orthogonal polynomials) to be monic  and the original form $v$ is proportional to $v_{0}$. This unique MOPS $\{Q_{n}(x)\}_{n\geqslant 0}$ with respect to the regular form $v_{0}$ can be characterized by the popular second order recurrence relation 
\begin{align} \label{MOPS rec rel} 
&	\left\{ \begin{array}{@{}l}
		Q_{0}(x)=1 \quad ; \quad Q_{1}(x)= x-\beta_{0} \vspace{0.15cm}\\
		Q_{n+2}(x) = (x-\beta_{n+1})Q_{n+1}(x) - \gamma_{n+1} \, Q_{n}(x) \ , \quad n\in\mathbb{N}_{0},  
	\end{array} \right. 
\end{align}
where $\beta_{n}=\frac{\langle v_{0},x Q_{n}^2  \rangle}{\langle v_{0}, Q_{n}^2  \rangle}$ and $\gamma_{n+1}=\frac{\langle v_{0}, Q_{n+1}^2  \rangle}{\langle v_{0}, Q_{n}^2  \rangle}$ for all $n\in\mathbb{N}_{0}$. For further considerations regarding the theory of orthogonal polynomials we refer for instance to \cite{ChiharaBook,IsmailBook}. 

A form $u$ is called {\it semiclassical} \cite{maroniSemiclassiques88, MaroniVariations} when it is regular and there exist two polynomials $\Phi$ and $\Psi$, with $\Phi$ monic and $\deg\Psi\geqslant 1$, such that \begin{equation}\label{eq semiclassical}
	(\Phi u)' + \Psi u=0 .
\end{equation} 
The pair $(\Phi,\Psi)$ is not unique. The previous equation can be simplified if and only if there exists a root $c$ of $\Phi$ such that 
\begin{equation}\label{cond Simplification semiclassical}
	\left| \Phi'(c) + \Psi(c)\right| + \left| <u,\theta_{c}^2 (\Phi) + \theta_{c}(\Psi)> \right| = 0\ , 
\end{equation}
where $\theta_{c}(f)(x)=\frac{f(x)-f(c)}{x-c}$, for any $f\in\mathcal{P}$, and $u$ would then fulfill \ $\left( \theta_{c}(\Phi) u \right)' + \left( \theta_{c}^2 (\Phi) + \theta_{c}(\Psi) \right) u =0$. 

The minimum value of the integer $s(\Phi,\Psi)=\max\left(\deg(\Phi)-2,\deg(\Psi)-1\right)$ taken among all the possible pairs $(\Phi,\Psi)$ in \eqref{eq semiclassical} is called the {\it class of the semiclassical form $u$} and denoted by $s$. The pair $({\Phi}, {\Psi})$ giving the class $s\geqslant 0$ is unique \cite{MaroniTheorieAlg,MaroniVariations} and is such that 
$$
	\prod_{c\in\mathcal{Z}_{\Phi}} 
	\left| \Phi'(c) + \Psi(c)\right| + \left| <u,\theta_{c}^2 (\Phi) + \theta_{c}(\Psi)> \right|
	\neq 0
$$
where $\mathcal{Z}_{\Phi}$ represents the set of the roots of $\Phi$. (For instance, the class of a semiclassical form $u$ is achieved if each simple root of $\Phi$ is also a root of $\Psi$ or if each double root of $\Phi $ is not a root of $\Psi$). 

By extension, the corresponding MOPS $\{P_{n}\}_{n\geqslant 0}$ of the semiclassical form $u$ of class $s$ is also called semiclassical of class $s$ and it always fulfills the structural relation 
$$
	\Phi(x) P_{n+1}'(x) = \sum_{\nu=n-s}^{n+t} \theta_{n,\nu} P_{\nu}(x) \ , \ n\geqslant s \ , \ t=\deg\Phi,
$$
with $\theta_{n,n-s}\theta_{n,n+t}\neq0, \  n\geqslant s$. Moreover, the elements of $\{P_{n}\}_{n\geqslant 0}$ are solution of a second order differential equation with polynomial coefficients depending on $n$ but of fixed degree for all $n\geqslant 0$. The classical polynomials (Hermite, Laguerre, Bessel and Jacobi) correspond to the case where $s=0$. 

Any affine transformation leaves invariant the orthogonality of a sequence, and so does the semiclassical character  \cite{MaroniTheorieAlg,MaroniVariations}.  Precisely,  $\{P_{n}\}_{n\geqslant 0}$  is orthogonal with respect to $u_{0}$ if and only if $\{\widetilde{P}_{n}\}_{n\geqslant 0}$, defined by $\widetilde{P}_{n}(x)=a^{-n} P_{n}(ax+b)$ with $a\neq0$,  is a MOPS with respect to 
$\widetilde{u}_{0}=\left(h_{a^{-1}} \circ\tau_{-b} \right) u_{0}$ where $h_{a}f(x)=f(ax)$ and $\tau_{-b}f(x)=f(x+b)$. 

\begin{lemma}\label{lem: affine transform} \cite{MaroniTheorieAlg,MaroniVariations} If $\{P_{n}\}_{n\geqslant 0}$ is a semiclassical MOPS with respect to $u_{0}$ satisfying \eqref{eq semiclassical}, then $\{\widetilde{P}_{n}\}_{n\geqslant 0}$ is also semiclassical with respect to $\widetilde{u}_{0}$ and it fulfills 
$$
	D\left( a^{-\deg\Phi} \Phi(ax+b) \widetilde{u}_{0} \right) + a^{1-\deg\Phi} \Psi(ax+b)\widetilde{u}_{0}=0. 
$$
\end{lemma}

On the other hand, the concept of orthogonality of a sequence has been broadened to the so-called $d$-orthogonality (and even more generally to the multiple orthogonality, a concept that we leave aside). 
Any sequence $\{P_{n}\}_{n\geqslant 0}$ is said to be $d$-orthogonal with respect to the vector functional $\mathbf{U}=(u_{0},\ldots, u_{d-1})^T$, if and only if it fulfills the following conditions \cite{MaroniDortho,Iseghem}
\begin{align}\label{dOPS def vector U}
	\left\{\begin{array}{lcl}
	\left< u_{k} , x^{m}P_{n}\right>=0 &,& n\geqslant md+k+1\ , \ m\geqslant 0, \\
	\left< u_{k} , x^{m}P_{md+k}\right>\neq0 &,&  m\geqslant 0. 
	\end{array}\right.
\end{align}
In this case, the $d$-MOPS $\{P_{n}\}_{n\geqslant 0}$ necessarily satisfies the $(d+1)$-order recurrence relation 
\begin{equation}\label{dOPS rec rel}
	P_{n+1}(x) = (x-\beta_{n}) P_{n}(x) - \sum_{\nu=0}^{d-1} \gamma_{n-\nu}^{d-1-\nu} P_{n-1-\nu}(x)
	\quad , \quad n\geqslant d+1, 
\end{equation}
where $\gamma_{n+1}^0 \neq 0$ for all $n\geqslant 0$, with the initial conditions 
\begin{equation}\label{dOPS initial cond}
	\left\{\begin{array}{l}
	P_{0}(x)=1 \quad ; \quad P_{1}(x)= x-\beta_{0} \\
	P_{n+1}(x)=  (x-\beta_{n}) P_{n}(x) - \sum_{\nu=0}^{n-1} \gamma_{n-\nu}^{d-1-\nu} P_{n-\nu-1-\nu}(x)
		\quad , \quad 1\leqslant n \leqslant d. 
	\end{array}\right.
\end{equation}

\section{The operator $KL_{\alpha}$ and some of its properties.} \label{Sec: KL alpha}

For practical reasons and to gather other cases sharing analogous properties, we are interested in dealing with a slight modification of the aforementioned KL-transform. The modification in case consists on a perturbation on the kernel of such transformation. Precisely, for $\alpha\geqslant 0$, let  
\begin{align}
	\label{KL alpha def}
	&	{  KL_{\alpha}[f(x)](\tau)} 
		= 2 \left| \Gamma\left( \alpha+1+\frac{i\tau}{2} \right)\right|^{-2}
		\int_{0}^{\infty} { x^{\alpha}}K_{i\tau}(2\sqrt{x}) f(x) dx \ , 
\\
	\label{KL alpha inverse}
	& 	{ x^{\alpha+1}}\ f(x)= \frac{1}{\pi^{2}} \lim_{\lambda\to \pi -} 
		\int_{0}^{\infty} \tau \sinh(\lambda\tau) \left| \Gamma\left( \alpha+1+\frac{i\tau}{2} \right)\right|^{2} 
		K_{i\tau}(2\sqrt{x})
		\ {  KL_{\alpha}[f](\tau)}d\tau \ .
\end{align}
which is valid for any continuous function  $f \in L_{1}\left(\mathbb{R}_{+}, K_{0}(2\mu\sqrt{x}) dx \right)$, $0<\mu<1$, in a neighborhood of each $x\in\mathbb{R}_{+}$ where $f(x)$ has bounded variation.

Naturally, the identity 
\begin{equation}\label{KL alpha xn}
	KL_{\alpha}[x^{n}](\tau)
	= \left(\alpha+1 -\frac{i \tau }{2}\right)_n \left(\alpha+1 +\frac{i \tau }{2}\right)_n
	= \left| \left( \alpha+1+\frac{i\tau}{2} \right)_{n}\right|^{2} 
\end{equation}
holds, enhancing the fact that $KL_{\alpha}$ is an isomorphism in the vector space $\mathcal{P}$, essentially  performing the passage between the canonical basis $\{x^{n}\}_{n\geqslant0}$ and the central factorial basis $\left\{\left| \left( \alpha+1+\frac{i\tau}{2} \right)_{n}\right|^{2} \right\}_{n\geqslant 0}$. When $\alpha=0$ or $\alpha=1/2$ we recover the central factorials treated in \cite{Riordan2} of even and odd order, respectively (see also \cite{LouYak2012,LouMarYak2011}). 

Besides, from the definition \eqref{KL alpha def}, we readily observe that 
\begin{equation}\label{prop KL alpha+gamma}
	KL_{\alpha+\beta}[f(x)](\tau)
	= \frac{ \left| \Gamma\left( \alpha+1+\frac{i\tau}{2} \right)\right|^{2}}
	{ \left| \Gamma\left( \alpha+\beta+1+\frac{i\tau}{2} \right)\right|^{2}} 
	KL_{\alpha}[x^{\beta} f(x)](\tau) , 
\end{equation}
and, in particular, when $\beta=n\in\mathbb{N}_{0}$, we have 
\begin{equation}\label{prop KL alpha+n}
	KL_{\alpha}[x^{n}f(x)](\tau)
	=  \left| \left( \alpha+1+\frac{i\tau}{2} \right)_{n}\right|^{2} KL_{\alpha+n}[f(x)](\tau)
	= KL_{\alpha}[x^{n}](\tau) \  KL_{\alpha+n}[f(x)](\tau) \ .
\end{equation}

As a matter of fact, the action of the $KL_{\alpha}$ operator acting on $\mathcal{P}$ can be viewed as the passage from differential relations into central difference relations, as it can be perceived from \eqref{KL alpha xn}.  To be more specific, let us represent the central difference operator by $\delta_{\omega}$, defined through 
\begin{equation}\label{delta w}
	(\delta_{\omega} f)(\tau):=\frac{f(\tau+\omega)-f(\tau-\omega)}{2\omega\tau}
\end{equation}
for some complex number $\omega\neq0$. 

\begin{lemma} For any $f\in\mathcal{P}$, the following identities hold  
\begin{equation}
		\left((\alpha+1)^{2}+\frac{\tau^{2}}{4}\right)\delta_{i}^{2} \left(  KL_{\alpha}[ f(x) ](\tau)\right)
	= KL_{\alpha}\left[ x \frac{d^{2}}{dx^{2}} f(x)\right] (\tau)
	\ , \ n\geqslant 0, 
\end{equation}
while 
\begin{equation}\label{KL alpha + 1/2}
	KL_{\alpha+1/2} \left[ \frac{d}{dx} f(x) \right] (\tau)
	= \delta_{i} \Big( KL_{\alpha}[ f(x) ] (\tau)\Big) \ . 
\end{equation}
\end{lemma}
\begin{proof} After some algebraic calculations, the action of the operator $\delta_{i}$ over the central factorials $(a+i\tau/2)_{n} (a-i\tau/2)_{n}$ brings 
\begin{equation}\label{delta i over factorials}
	\delta_{i} \Big( (a+i\tau/2)_{n} (a-i\tau/2)_{n} \Big)
	= n  (a+1/2+i\tau/2)_{n-1} (a+1/2-i\tau/2)_{n-1} 
\end{equation}
and therefore 
\begin{equation*}
	\delta_{i}^{2} \Big( (a+i\tau/2)_{n} (a-i\tau/2)_{n} \Big)
	= n(n-1)  (a+1+i\tau/2)_{n-2} (a+1-i\tau/2)_{n-2} , 
\end{equation*}
or, equivalently,  
$$
	\left(a^{2}+\frac{\tau^{2}}{4}\right)\delta_{i}^{2} \Big( (a+i\tau/2)_{n} (a-i\tau/2)_{n} \Big) = n(n-1)  (a+i\tau/2)_{n-1} (a-i\tau/2)_{n-1} \ , \ n\geqslant 1. 
$$
While from this latter equality we may read 
$$
	\left((\alpha+1)^{2}+\frac{\tau^{2}}{4}\right)\delta_{i}^{2} \left(  KL_{\alpha}[ x^{n} ](\tau)\right)
	= KL_{\alpha}\left[ x \frac{d^{2}}{dx^{2}} x^{n}\right] (\tau)
	\ , \ n\geqslant 0, 
$$
the relation \eqref{delta i over factorials} provides 
\begin{equation*}
	KL_{\alpha+1/2} \left[ \frac{d}{dx} x^{n}\right] (\tau)
	= \delta_{i} \Big(KL_{\alpha}[x^{n}] (\tau)\Big) \ , \ n\in\mathbb{N}_{0}. 
\end{equation*}

%

The result now follows because of the fact that  $\{x^{n}\}_{n\geqslant 0}$ forms a basis of $\mathcal{P}$. 
\end{proof}

Following the results in  \cite[Ch. 2]{YakubovichBook1996}, we have the following analog of the Plancherel theorem for transform \eqref{KL alpha def}. 
\begin{theorem}\label{Thm: Parseval orginal} The operator $KL_{\alpha}$ is an isomorphism between Hilbert spaces 
$$
	KL_{\alpha}: \ L_{2}(\mathbb{R}_{+};x^{2\alpha+1} dx ) \rightarrow 
		L_{2}\left(\mathbb{R}_{+}; \tau \sinh(\pi \tau) 
	\left| \Gamma\left(\alpha+1+\tfrac{i\tau}{2}\right)\right|^{4} \frac{d\tau}{4\pi^{2}}\right) 
$$
where integral \eqref{KL alpha def} is understood in the mean square sense with respect to the norm in the image space. Reciprocally, the inverse operator has the form \eqref{KL alpha inverse} with the corresponding integral in the mean square sense by the norm in $L_{2}(\mathbb{R}_{+};x^{2\alpha+1} dx ) $ and the following generalized Parseval equality holds 
\begin{equation}\label{Parseval}
	\int_{0}^{\infty} x^{2\alpha+1} f(x) g(x) dx 
	= \frac{1}{4\pi^{2}} \int_{0}^{\infty} \tau \sinh(\pi \tau) 
	\left| \Gamma\left(\alpha+1+\frac{i\tau}{2}\right)\right|^{4} KL_{\alpha}[f](\tau)KL_{\alpha}[g](\tau) d\tau \  ,
\end{equation}
where $f,g\in L_{2}(\mathbb{R}_{+};x^{2\alpha+1} dx ) $. 
\end{theorem}

Despite the underlying conditions in the latter theorem, \eqref{Parseval} remains nevertheless valid if one of the functions is a polynomial, as long as further conditions are assumed over the other function. 

\begin{proposition}\label{Prop: Parseval for polys} For any $f\in\mathcal{P}$ and any nonnegative function $g\in L_{2}(\mathbb{R}_{+};x^{2\alpha+1} dx ) $, we have  
\begin{equation}\label{Parseval for polys}
	\int_{0}^{\infty} x^{2\alpha+1} f(x) g(x) dx 
	= \frac{1}{4\pi} \int_{0}^{\infty}  KL_{\alpha}[f](\tau) 
	\frac{
	\left| \Gamma\left(\alpha+1+\frac{i\tau}{2}\right)\right|^{4}}{\left| \Gamma(i\tau)\right|^{2}} KL_{\alpha}[g](\tau) d\tau \  ,
\end{equation}
as long as $KL_{\alpha}[g] \in L_{1} (\mathbb{R}_{+} \ , \ \e^{(\frac{\pi}{2}-\delta)\tau} (1+\tau)^{2(\alpha+1)} d\tau)$ for some $\delta \in (0,\pi/2)$.
\end{proposition}

\begin{proof} 
 Consider for  sufficiently small positive $\epsilon$ and $n\in\mathbb{N}_{0}$, $f_{\epsilon}(x)=\e^{-2\epsilon\sqrt{x}}x^{n}\to x^{n}$, as $\epsilon\to0$. On the grounds of Theorem \ref{Thm: Parseval orginal}, the identity  \eqref{Parseval} yields  
\begin{equation} \label{proof Parseval g any}
		 \int_{0}^{\infty} x^{2\alpha+1} f_{\epsilon}(x) g(x) dx 
	= \frac{1}{4\pi} \int_{0}^{\infty}  KL_{\alpha}[f_{\epsilon}](\tau) 
	\frac{
	\left| \Gamma\left(\alpha+1+\frac{i\tau}{2}\right)\right|^{4}}{\left| \Gamma(i\tau)\right|^{2}} KL_{\alpha}[g](\tau) d\tau \ .
\end{equation}
Let 
$$
	\varphi_{\epsilon}(\tau) :=\frac{
	\left| \Gamma\left(\alpha+1+\frac{i\tau}{2}\right)\right|^{4}}{\left| \Gamma(i\tau)\right|^{2}} 
	KL_{\alpha}\left[ f_{\epsilon}(x) \right](\tau) 
	= \frac{
	\left| \Gamma\left(\alpha+1+\frac{i\tau}{2}\right)\right|^{2}}{\left| \Gamma(i\tau)\right|^{2}} 
	\frac{1}{2^{2(n+\alpha)}}\int_{0}^{\infty} x^{2(n+\alpha)+1} \e^{-\epsilon x} K_{i\tau}(x) dx.
$$
Taking into account inequality \eqref{ineq Kitau} (with $m=0$) along with the Stirling asymptotic formula for the Gamma function we obtain the estimate 
$$
	\left| \varphi_{\epsilon}(\tau) \right| 
	\leqslant C_{n,\delta}\, \e^{(\frac{\pi}{2}-\delta) \tau} |\tau|^{2(\alpha+1)} \int_{0}^{\infty}x^{2(n+\alpha)+1} K_{0}(x\cos \delta) dx
	= O( \e^{(\frac{\pi}{2}-\delta) \tau} \tau^{2(\alpha+1)}),\  \tau\to +\infty,
$$
for some $\delta\in(0,\pi/2)$ and $C_{n,\delta}$ does not depend on $\epsilon$. 
Thus, as long as the function $g$ satisfies the claimed conditions, it readily follows 
\begin{equation}\label{proof Parseval rel2}
	\int_{0}^{\infty} \left| \varphi_{\epsilon}(\tau) \right| \left| 
	KL_{\alpha}[g](\tau)\right|  d\tau
	< C_{\alpha} , 
\end{equation}
where $C_{\alpha}>0$ does not depend on $\epsilon$. Therefore the integral on the right hand side of \eqref{proof Parseval g any} converges absolutely and uniformly as $\epsilon\to 0+$ by virtue of Weierstrass test.

On the other hand, via Fatou lemma along with  \eqref{proof Parseval g any}-\eqref{proof Parseval rel2}, we have 
$$
	 \int_{0}^{\infty}  x^{2\alpha+1+n}  g(x)  dx
	 \leqslant 
	\lim_{\epsilon\to0}  \int_{0}^{\infty}  x^{2\alpha+1} f_{\epsilon}(x)  g(x)  dx < C_{\alpha}.
$$
Moreover, since $f_{\epsilon}(x)$ is a monotone increasing as $\epsilon$ approaches $0$, by Levi's theorem, the passage to the limit under the integral signs through \eqref{proof Parseval g any} is allowed. 

The result now follows due to the fact that $\lim_{\epsilon\to0}KL_{\alpha}[f_{\epsilon}](\tau)=KL_{\alpha}[f](\tau)$ and also because $\{x^{n}\}_{n\geqslant0}$ forms a basis of $\mathcal{P}$. 
\end{proof}

In particular, we have: 

\begin{example}\label{Example: Parseval Exponential} The particular choice of $g(x)=x^{\beta-\alpha-1}\e^{-x}$, which belongs to $L_{2}(\mathbb{R}_{+};x^{2\alpha+1} dx ) $ as long as $\alpha+\beta+1>0$. On the other hand, we also have  $KL_{\alpha}\left[ g(x) \right](\tau) \in L_{1} (\mathbb{R}_{+} \ , \ \e^{(\frac{\pi}{2}-\delta)\tau} (1+\tau)^{2(\alpha+1)} d\tau)$ for some $\delta \in (\pi/4,\pi/2) $ inasmuch as, on the grounds of relation (2.16.8.4) in \cite{PrudnikovMarichev}, we have 
$$
	\left| \Gamma\left(\alpha+1+\frac{i\tau}{2}\right)\right|^{2} KL_{\alpha}[g](\tau)
	=  \sqrt{e} \left| \Gamma\left(\beta+\frac{i\tau}{2}\right)\right|^{2} 
	W_{-\beta+\frac{1}{2},\frac{i\tau}{2}}(1)
		= O\left( \e^{\pi \tau/4}\tau^{\beta}\right) \ , \ \tau\to +\infty \ , 
$$
considering the Stirling asymptotic formula for the Gamma function together with the asymptotic expansion by index $\tau$ for the Whittaker function $W_{\mu,i\tau}(x)$ \cite[p.25]{YakubovichBook1996}.
In the light of Proposition \ref{Prop: Parseval for polys}, the equality 
\begin{align} \label{Parseval Exponential}
	\ds \int_{0}^{\infty} x^{\alpha+\beta} \e^{-x} f(x) dx 
	= \frac{\sqrt{\e}}{4\pi} \int_{0}^{\infty} 
	KL_{\alpha}[f](\tau)  \frac{
	\left|\Gamma\left(\alpha+1+\frac{i\tau}{2}\right)  
	\Gamma\left(\beta +\frac{i\tau}{2}\right)\right|^{2}}
	{\left|\Gamma\left( i\tau \right)\right|^{2}}  
	W_{-\beta+\frac{1}{2},\frac{i\tau}{2}}(1)
	 d\tau \ , 
\end{align}
holds, for any $f\in\mathcal{P}$ and $\beta>0$. 
\end{example}

\begin{corollary} The identity \eqref{Parseval for polys} remains valid for any complex valued function $g\in L_{2}(\mathbb{R}_{+};x^{2\alpha+1} dx ) $ such that 
$KL_{\alpha}\left[\mathop{{\rm Re}(g)}_{\left\{\mathop{}{+ \atop -}\right\} }\right],
KL_{\alpha}\left[\mathop{{\sf Im}(g)}_{\left\{\mathop{}{+ \atop -}\right\} }\right]
 \in L_{1} (\mathbb{R}_{+} \ , \ \e^{(\frac{\pi}{2}-\delta)\tau} (1+\tau)^{2(\alpha+1)} d\tau)$ for some $\delta \in (0,\pi/2)$, where $h_{+}=\max(0,h)$ and $h_{-}=-\min(0,h)$ denote the positive and negative parts of $h$.   
\end{corollary}

\begin{proof}  We write $h=h_{+}-h_{-}$ and then we apply the preceding Proposition \ref{Prop: Parseval for polys} to $h_{+}$ and $h_{-}$.  
\end{proof}
 
In general, the conditions on $KL_{\alpha}\left[\mathop{{\rm Re}(g)}_{\left\{\mathop{}{+ \atop -}\right\} }\right],
KL_{\alpha}\left[\mathop{{\sf Im}(g)}_{\left\{\mathop{}{+ \atop -}\right\} }\right]$ are difficult to verify. However, there are some cases where we are able to compute $KL_{\alpha}[g]$, avoiding further inspections over the real and imaginary parts.

\begin{corollary}\label{prop: Parseval} For any $f\in\mathcal{P}$ and $|{\sf Im} \mu|<2\beta$,  the following identity takes place  
\begin{align}\label{rel Parseval Kitau}
\begin{array}{l}		
	\ds \int_{0}^{\infty} x^{\alpha+\beta} f(x) K_{i\mu}(2\sqrt{x})dx \\
	\ds = \frac{1}{8\pi\Gamma(2\beta) } \int_{0}^{\infty} KL_{\alpha}[f](\tau)
		 \frac{\left| 
		\Gamma\left(\beta+\frac{i(\tau+\mu)}{2}\right)
		\Gamma\left(\beta+\frac{i(\tau-\mu)}{2}\right)  
		\Gamma\left(\alpha+1+\frac{i\tau}{2}\right)\right|^{2}}
		{ \left| \Gamma\left(i\tau\right)\right|^{2}}
		d\tau \ .
\end{array}
\end{align}
\end{corollary}
\begin{proof} 
Consider for  sufficiently small positive $\epsilon$ and $n\in\mathbb{N}_{0}$, $f_{\epsilon}(x)=\e^{-2\epsilon\sqrt{x}}x^{n}$ and $g(x)= x^{\beta-\alpha-1} K_{i\mu}(2\sqrt{x})$, with  $\beta>0$. According to Theorem \ref{Thm: Parseval orginal}, the identity  \eqref{Parseval} yields  
\begin{equation} \label{proof Parseval epsilon K}
		 \int_{0}^{\infty} x^{\alpha+\beta} f_{\epsilon}(x) K_{i\mu}(2\sqrt{x}) dx 
	=\frac{1}{4\pi^{2}} \int_{0}^{\infty} \tau \sinh(\pi \tau) 
	\left| \Gamma\left(\alpha+1+\frac{i\tau}{2}\right)\right|^{4} KL_{\alpha}[f_{\epsilon}](\tau)KL_{\alpha}[g](\tau) d\tau \ .
\end{equation}
On the grounds of relation (2.16.33.2) in \cite{PrudnikovMarichev}, we have 
$$
	\left| \Gamma\left(\alpha+1+\frac{i\tau}{2}\right)\right|^{2} KL_{\alpha}\left[ g(x) \right](\tau) 
	=  \frac{\left| 
	\Gamma\left(\beta+\frac{i(\tau+\mu)}{2}\right)
	\Gamma\left(\beta+\frac{i(\tau-\mu)}{2}\right) \right|^{2}}
	{2\Gamma(2\beta) } \ . 
$$
Let 
$$
	\varphi_{\epsilon}(\tau) :=\left| \Gamma\left(\alpha+1+\frac{i\tau}{2}\right)\right|^{2} KL_{\alpha}\left[ f_{\epsilon}(x) \right](\tau) 
	= \frac{1}{2^{2(n+\alpha)}}\int_{0}^{\infty} x^{2(n+\alpha)+1} \e^{-\epsilon x} K_{i\tau}(x) dx.
$$
Taking into account inequality \eqref{ineq Kitau} (with $m=0$) we obtain the estimate 
$$
	\left| \varphi_{\epsilon}(\tau) \right| 
	\leqslant C_{n,\delta}\, \e^{-\delta \tau} \int_{0}^{\infty}x^{2(n+\alpha)+1} K_{0}(x\cos \delta) dx
	= O(e^{-\delta \tau})
$$
where $\delta\in(0,\pi/2)$ and $C_{n,\delta}$ does not depend on $\epsilon$. Meanwhile, taking into account the Stirling asymptotic formula for Gamma functions  \cite[Vol.I]{Bateman}, we find 
$$
	\frac{\left| 
	\Gamma\left(\beta+\frac{i(\tau+\mu)}{2}\right)
	\Gamma\left(\beta+\frac{i(\tau-\mu)}{2}\right) \right|^{2}}
	{\left| 
	\Gamma\left(i\tau\right) \right|^{2} }
	= O(\tau^{4\beta-1}) \ , \ \tau\to +\infty. 
$$
Consequently, the integral in the right hand side of \eqref{proof Parseval epsilon K} converges absolutely and uniformly by  Weierstrass' test, since 
$$
	\int_{0}^{\infty} \left| \varphi_{\epsilon}(\tau) \right| \frac{\left| 
	\Gamma\left(\beta+\frac{i(\tau+\mu)}{2}\right)
	\Gamma\left(\beta+\frac{i(\tau-\mu)}{2}\right) \right|^{2}}
	{\left| 
	\Gamma\left(i\tau\right) \right|^{2} } d\tau
	= O\left( \int_{0}^{\infty} e^{-\delta \tau} (1+ \tau)^{4\beta-1} d\tau  \right) < \infty. 
$$
On the other hand, the left hand side of the equality \eqref{proof Parseval epsilon K} converges absolutely and uniformly via the straightforward estimate 
$$
	\int _{0}^{\infty} x^{\alpha+\beta} \left| f_{\epsilon}(x) K_{i\mu}(2\sqrt{x})\right| dx
	\leqslant 
	\int _{0}^{\infty} x^{\alpha+\beta+n}  K_{{\sf Im} \mu}(2\sqrt{x}) dx < \infty 
	\ , \ n\geqslant \mathbb{N}_{0}, 
$$
valid under the condition  $|{\sf Im} \mu|< 2(\alpha+\beta+1)$, which is motivated by the asymptotic \eqref{Knu at infty}-\eqref{K0 at 0} of the modified Bessel function. 

Thus, passing to the limit $\epsilon \to 0+$ through equality \eqref{proof Parseval epsilon K} and using \eqref{KL alpha xn}, we come out with the identities 
\begin{equation}\label{proof Parseval Kitau Gammas id} 
	\ds \int_{0}^{\infty} x^{n+\alpha+\beta} K_{i\mu}(2\sqrt{x}) dx 
	=\ds  \int_{0}^{\infty} 
	  \frac{\left| 
	\Gamma\left(\beta+\frac{i(\tau+\mu)}{2}\right)
	\Gamma\left(\beta+\frac{i(\tau-\mu)}{2}\right) \Gamma\left(\alpha+n+1+\frac{i\tau}{2}\right) \right|^{2}}
	{8\pi \Gamma(2\beta) \left| \Gamma(i\tau)\right|^{2}}d\tau  \  ,
\end{equation}
which naturally imply \eqref{rel Parseval Kitau} insofar as $\{x^{n}\}_{n\geqslant 0}$ forms a basis of $\mathcal{P}$. 
\end{proof}

This latter result can be actually interpreted in terms of a integral representation for the $KL_{\alpha}$-transform of a polynomial sequence:  

\begin{remark} For any $\beta,\mu>0$ and $f\in\mathcal{P}$, we have 
\begin{equation}\label{KL alpha of x beta Parseval}
\begin{array}{lcl}
	\ds KL_{\alpha+\beta}[ f(x)](\mu)
	&=& \ds  \frac{1}{4\pi\Gamma(2\beta)\left| \Gamma\left(\alpha+\beta+1+\frac{i\mu}{2}\right)\right|^{2}} \\
	&\times&	\ds \int_{0}^{\infty} KL_{\alpha}[f](\tau)
		 \frac{\left| 
		\Gamma\left(\beta+\frac{i(\tau+\mu)}{2}\right)
		\Gamma\left(\beta+\frac{i(\tau-\mu)}{2}\right)  
		\Gamma\left(\alpha+1+\frac{i\tau}{2}\right)\right|^{2}}
		{ \left| \Gamma\left(i\tau\right)\right|^{2}}
		d\tau \ .
\end{array}
\end{equation}	
\end{remark}

Considering in Proposition \ref{prop: Parseval} the exponential function instead of the modified Bessel function, we can as well find a Parseval-like relation also valid when one of the integrands is a polynomial.

\begin{remark} 
Upon the choice of  $f(x)=x^{n}$, $n\geqslant 0$, in \eqref{rel Parseval Kitau}, we obtain \eqref{proof Parseval Kitau Gammas id}, 
while  \eqref{Parseval Exponential} yields  
\begin{equation}\label{proof Parseval epsilon exp rel2}
	 \int_{0}^{\infty} x^{\alpha+\beta+n}  \e^{-x} dx 
	=\frac{1}{4\pi} \int_{0}^{\infty}  KL_{\alpha}[x^{n}](\tau)  \frac{
	\left|\Gamma\left(\alpha+1+\frac{i\tau}{2}\right)  
	\Gamma\left(\beta+\frac{i\tau}{2}\right)\right|^{2}}
	{\left|\Gamma\left( i\tau \right)\right|^{2}}  
	W_{ -\beta+\frac{1}{2},\frac{i\tau}{2}}(1)
	 d\tau \  ,\  n\geqslant 0.
\end{equation}

Both equalities can be read as integral relations between the Gamma functions, precisely  \eqref{proof Parseval Kitau Gammas id} can be rewritten as 
\begin{eqnarray*}
	&& \frac{1}{2}\left| \Gamma\left( \alpha+\beta+n + \frac{i\mu}{2}\right) 
	\Gamma\left( \alpha+\beta + \frac{i\mu}{2}\right)\right|^{2} \\
	&& = \frac{1}{8\pi \Gamma(2\beta)} \int_{0}^{\infty} 
	  {\left| 
	\Gamma\left(\beta+\frac{i(\tau+\mu)}{2}\right)
	\Gamma\left(\beta+\frac{i(\tau-\mu)}{2}\right) \Gamma\left(\alpha+n+1+\frac{i\tau}{2}\right) \right|^{2}}
	\frac{d\tau}{\left| \Gamma(i\tau)\right|^{2}}
	 \  , \  n\geqslant 0,
\end{eqnarray*}
whereas, from \eqref{proof Parseval epsilon exp rel2} we derive 
\begin{equation*}
	\Gamma(\alpha+\beta+n)
	=\frac{1}{4\pi} \int_{0}^{\infty}   \frac{
	\left|\Gamma\left(\alpha+n+1+\frac{i\tau}{2}\right)  
	\Gamma\left(\beta+\frac{i\tau}{2}\right)\right|^{2}}
	{\left|\Gamma\left( i\tau \right)\right|^{2}}  
	W_{ -\beta+\frac{1}{2},\frac{i\tau}{2}}(1)
	 d\tau \  , \  n\geqslant 0. 
\end{equation*}
As far as we are concerned the two latter identities are new. 
\end{remark}

Another important fact related to the $KL_{\alpha}$-transform lies on the fact that the Bessel function is an eigenfunction of the differential operator $\mathcal{A}$ in \eqref{op A }. The following result is a key ingredient for the achievements in \S\ref{sec: KL MOPS to dMOPS}. 

\begin{lemma}\label{lem:KL alpha differential} For any $m,n\in\mathbb{N}_{0}$ and any $f\in\mathcal{P}$, it is valid 
\begin{equation}
	\label{KL alpha differential}
	KL_{\alpha}\left[ \left( \frac{1}{x} \mathcal{A} x + 2\alpha \frac{d}{dx} x \right)^{m} x^{n}f(x)\right](\tau)
	= (-1)^{m}\left( \frac{\tau^{2}}{4} + \alpha^{2} \right)^{m}
	\left| \left( \alpha+1+\frac{i\tau}{2} \right)_{n}\right|^{2} KL_{\alpha+n}[f](\tau) , 
\end{equation}
where $\mathcal{A}$ represents the operator \eqref{op A }.
\end{lemma}

\begin{proof} 
The following identity 
\begin{equation}\label{int by parts}
	\int_{0}^\infty \psi(x) \Big(\mathcal{A}\varphi(x)\Big) \frac{dx}{x}
	=	\int_{0}^\infty  \Big(\mathcal{A} \psi(x)\Big) \varphi(x) \frac{dx}{x},
\end{equation}
is valid whenever $\phi,\psi$ are two functions in $\mathcal{C}_{0}^{2}(\mathbb{R}_{+})$ vanishing at infinity and also near the origin together with their derivatives, in order to eliminate the outer terms. 
Bearing in mind 
$$
	x^{\alpha} \left(\frac{1}{x} \mathcal{A} x + 2\alpha \frac{d}{dx} x \right) f(x)
	= \frac{1}{x} \mathcal{A} x  \left[x^{\alpha} f(x)\right]  - \alpha^{2} x^{\alpha} f(x)
$$
then, relations \eqref{int by parts} and \eqref{A Kitau} ensure 
$$ 
	\int_{0}^\infty x^{\alpha}K_{i\tau}(2\sqrt{x}) \left( \left(\frac{1}{x} \mathcal{A} x + 2\alpha \frac{d}{dx} x \right) f(x) \right) dx 
	= - \left(  \frac{\tau^{2}}{4}+ \alpha^{2} \right)  
		 \int_{0}^\infty x^{\alpha}K_{i\tau}(2\sqrt{x})  f(x)  dx \ .
$$
The relation \eqref{KL alpha differential} now follows if we proceed by finite induction over $m$ and take into account \eqref{prop KL alpha+n}. 
\end{proof}

Inasmuch as 
$$
	\frac{1}{x} \mathcal{A} x + 2\alpha \frac{d}{dx} x 
	= x \left( \frac{d}{dx}x\frac{d}{dx} + 2(\alpha+1) \frac{d}{dx} -1 \right) + 2\alpha +1,
$$ 
the relation  \eqref{prop KL alpha+gamma} enables  
\begin{eqnarray*}
	KL_{\alpha}\left[ \left( \frac{1}{x} \mathcal{A} x + 2\alpha \frac{d}{dx} x \right) f(x)\right]
	&=& \left((\alpha+1)^{2}+\frac{\tau^{2}}{4}\right) KL_{\alpha+1} \left[ \left( \frac{d}{dx}x\frac{d}{dx} + 2(\alpha+1) \frac{d}{dx} -1 \right)  
			f(x)\right]\\
	&& + (2\alpha+1) KL_{\alpha} [f(x)](\tau)
\end{eqnarray*}
which, because of \eqref{KL alpha differential} with $(m,n)=(1,0)$ yields 
\begin{equation} \label{KL alpha to KL alpha +1} 
	- KL_{\alpha}[f(x)](\tau) 
	= KL_{\alpha+1} \left[ \left( \frac{d}{dx}x\frac{d}{dx} + 2(\alpha+1) \frac{d}{dx} -1 \right)  f(x)\right] (\tau) \ .  
\end{equation}
Hence, by induction, it follows 
\begin{equation} \label{KL alpha to KL alpha +1} 
	(-1)^{m} KL_{\alpha}[f(x)](\tau) 
	= KL_{\alpha+m} \left[ \prod_{\sigma=0}^{m}\left( \frac{d}{dx}x\frac{d}{dx} + 2(\alpha+\sigma+1) \frac{d}{dx} -1 \right)  f(x)\right] (\tau) \ , 
	\ m\in\mathbb{N}_{0}.
\end{equation}

\begin{remark} 
Mimicking the analysis taken in \cite{LouYak2012}, on the strength of relation \eqref{KL alpha differential}, the MPS $\{P_{n}\}_{n\geqslant0}$ defined by 
$$
	P_{n}(x;\alpha):=  (-1)^{n}\left( \frac{1}{x} \mathcal{A} x + 2\alpha \frac{d}{dx} x \right)^{n} 
				= (-1)^{n}  \Big(\left( \frac{d}{dx}x \frac{d}{dx} + 2\alpha \frac{d}{dx} -1 \right)x\Big)^{n}
				\ , \ n\geqslant 0,
$$
is such that $KL_{\alpha}[P_{n}(x;\alpha)]=\left(\frac{\tau^{2}}{4}+\alpha^{2}\right)^{n}$ for $n\in\mathbb{N}_{0}$. 

The polynomial basis $\left\{\left(\frac{\tau^{2}}{4}+\alpha^{2}\right)^{n}\right\}_{n\geqslant0}$ and the factorial basis $\{(\alpha+1-\frac{i\tau}{2})_{n}(\alpha+1+\frac{i\tau}{2})_{n}\}_{n\geqslant0}$ are bridged via 
the two sets of numbers $\{t_{n,\nu}(\alpha)\}_{0\leqslant \nu\leqslant n}$ and $\{T_{n,\nu}(\alpha)\}_{0\leqslant \nu\leqslant n}$:
\begin{equation}\label{CentralFact coeffs}
\begin{array}{ll}
	& \ds  \left(\frac{\tau^{2}}{4}+\alpha^{2}\right)^{n} 
	 = \sum_{\nu=0}^{n}T_{n,\nu}(\alpha) \left|\left(\alpha+1+\frac{i\tau}{2}\right)_{\nu}\right|^{2} \\
\text{ whereas }  & \\
	&\ds  \left|\left(\alpha+1+\frac{i\tau}{2}\right)_{n}\right|^{2}
	 = \sum_{\nu=0}^{n}t_{n,\nu}(\alpha) \left(\frac{\tau^{2}}{4}+\alpha^{2}\right)^{\nu}
	 \ , \ n\geqslant 0. 
	 \end{array}
\end{equation}
 These are essentially the non-centered central factorial numbers (or modified Stiring numbers) defined by the triangular recurrence relations 
$$
	\widehat{t}_{n,n}(\alpha)=1 
	\quad ; \quad  
	\widehat{t}_{n+1,\nu}(\alpha) 
		=\widehat{t}_{n,\nu-1}(\alpha)+ (2\alpha+n+1)(n+1) \ \widehat{t}_{n,\nu}(\alpha) 
		 \ , \ 0\leqslant \nu \leqslant n, 
$$
with $t_{n,-1}(\alpha)=0$, while 
$$
	\widehat{T}_{n,n}(\alpha)=1 
	\quad ; \quad  
	\widehat{T}_{n+1,\nu}(\alpha) 
		=  \widehat{T}_{n,\nu-1}(\alpha) - (2\alpha+\nu+1)(\nu+1) \ \widehat{T}_{n,\nu}(\alpha) 
		 \ , \ 0\leqslant \nu \leqslant n, 
$$
with $T_{n,-1}(\alpha)=0$. 

Recalling the expressions of the so-called $2\alpha$-modified Stirling numbers of first kind $\widehat{s}_{2\alpha}(n,\nu)$ and of second kind $\widehat{S}_{2\alpha}(n,\nu)$ treated in \cite{Loureiro2010}, we conclude that 
$$
	t_{n,\nu}(\alpha)=(-1)^{n+\nu}\widehat{s}_{2\alpha}(n+1,\nu+1)
	\quad \text{whereas}\quad 
	T_{n,\nu}(\alpha)=(-1)^{n+\nu}\widehat{S}_{2\alpha}(n+1,\nu+1)
	\ , \ 0\leqslant \nu \leqslant n .
$$
Therefore, on the grounds of \eqref{KL alpha differential}, necessarily, 
$$
	P_{n}(x;\alpha)= \sum_{\nu=0}^{n} (-1)^{n+\nu} \widehat{S}_{2\alpha}(n+1,\nu+1) x^{\nu} \ , \ n\geqslant 0.
$$

Moreover, from the relation \eqref{KL alpha differential} with $f(x)=1$, it follows
$$
	KL_{\alpha}\left[ \left( \frac{1}{x} \mathcal{A} x + 2\alpha \frac{d}{dx} x \right)^{m} x^{n}\right](\tau)
	= 	\sum_{\nu=0}^{n} (-1)^{m+n+\nu}\widehat{s}_{2\alpha}(n+1,\nu+1) (\alpha) 
		\left( \frac{\tau^{2}}{4} + \alpha^{2} \right)^{m+\nu} \ , \ n,m\geqslant 0,
$$
which, due to the linearity and injectivity of the operator $KL_{\alpha}$, provides 
$$
	\left( \frac{1}{x} \mathcal{A} x + 2\alpha \frac{d}{dx} x \right)^{m} x^{n}
	= \sum_{\nu=0}^{n} (-1)^{m+n+\nu}\widehat{s}_{2\alpha}(n+1,\nu+1) (\alpha) 
		P_{m+\nu}(x;\alpha) \ , \ n,m\geqslant 0.
$$
\end{remark}

\begin{remark} The $KL_{\alpha}$-transform can be extended to any continuous function $f \in L_{1}\left(\mathbb{R}, K_{0}(2\mu\sqrt{|x|}) dx \right)$, $0<\mu<1$, in a neighborhood of each $x\in\mathbb{R}$ where $f(x)$ has bounded variation in the following manner 
$$
	KL_{\alpha}[f(x)](\tau) = \frac{2\sinh(\pi {\tau}/2)}{\pi{\tau}}
		\int_{-\infty}^{+\infty} K_{2i {\tau}}(2\sqrt{|x|}) f(|x|) |x|^{\alpha} dx
$$
Consequently, the latter results can be extended to the vector space of polynomials $\mathcal{P}$ without any further restrictions over the domain.
\end{remark}

\section{Generalized hypergeometric-type polynomials. Illustrative examples. }\label{Sec: Hypergeom polys}

Despite its simplicity, equality \eqref{KL alpha xn} permits to induce structural relations known from one  MPSs to the other. 
This argument is substantiated by a few examples as we will soon see: when $\{B_{n}\}_{n\geqslant 0}$ is either the Hermite or the Laguerre polynomial sequences then its $KL_{\alpha}$-transform  is, respectively, a $4$ and $2$-orthogonal MOPS, whose structure is completely unraveled.

The $KL_{\alpha}$-transform of a MPS $\{P_{n}\}_{n\geqslant 0}$ of hypergeometric type defined by 
\begin{equation}\label{poly pFq}
	P_{n}(x) = (-1)^{n}\frac{\left(\prod_{\nu=1}^{q} (b_{\nu})_{n}\right)}
			{\left(\prod_{\nu=1}^{p} (a_{\nu})_{n}\right)}\ 
			{}_{p+1}F_{q}\left(\left.\begin{array}{c}{-n,a_{1},\ldots,a_{p}}\\
						{b_{1},\ldots,b_{q}}\end{array}\right|  x \right)
						\ , \ n\geqslant 0, 
\end{equation}
where the coefficients $a_{j},b_{k}$ with $j=1,\ldots,p$ and $k=1,\ldots ,q$, do not depend on $x$ but possibly depending on $n$, is again an hypergeometric polynomial type MPS, say $\{S_{n}\}_{n\geqslant 0}$
and is given by 
\begin{equation}\label{poly pFq image}
	S_{n}(\tau^{2}/4) = (-1)^{n}\frac{\left(\prod_{\nu=1}^{q} (b_{\nu})_{n}\right)}
			{\left(\prod_{\nu=1}^{p} (a_{\nu})_{n}\right)}\ 
			 {}_{p+3}F_{q}\left(\left.\begin{array}{c}{-n,a_{1},\ldots,a_{p},\alpha+1-\frac{i\tau}{2},\alpha+1+\frac{i\tau}{2}}\\
						{b_{1},\ldots,b_{q}}\end{array}\right|  1 \right)
						\ , \ n\geqslant 0.
\end{equation}

The description of all the hypergeometric orthogonal polynomials of this type hierarchized in the Askey table,  in particular those of the form \eqref{poly pFq}. For instance, a table describing all the $d$-orthogonal polynomial sequences of hypergeometric type have been set up in  \cite{BenCheikhTabledOrtho}, the $d$-orthogonal sequences that will soon arise are presented therein. 

While on  \S\ref{subsec: Appell}  the $KL_{\alpha}$-image of any Appell $d$-orthogonal sequence is descried (the Hermite polynomials included), on \S\ref{subsect: Rev Appell} upon the description of all the $d$-orthogonal sequences whose reversed sequence is of Appell, the Continuous Dual Hahn polynomials will arise as the $KL_{\alpha}$-image of a $2$-orthogonal polynomial sequence associated to the Bateman's function studied in \cite{BCheikhDouak2000}.

\subsection{The Appell polynomials} \label{subsec: Appell}

A MPS $\{P_{n}\}_{n\geqslant 0}$ is said to be an Appell MPS, whenever \cite{Appell}
\begin{equation}\label{Appell MPS}
	\frac{d}{dx} P_{n+1}(x) = (n+1)P_{n}(x) \ , \ n\in\mathbb{N}_{0}. 
\end{equation}

As described in \cite{DouakAppell}, the unique polynomial sequences that are simultaneously d-orthogonal, d-symmetric and of Appell are the $d$-OPS of type Hermite, $\{\widehat{H}_{n}\}_{n\geqslant0}$ fulfilling the recursive relation of order $(d+1)$ 
\begin{equation}
	\begin{array}{l}
	\ds \widehat{H}_{n+d+1}(x;d) = x \widehat{H}_{n+d}(x;d)
			- (d+1)^{-1} \binom{n+d}{d} \widehat{H}_{n}(x;d)\\
	\ds \widehat{H}_{n}(x;d) = x^{n} \ , \ n=0,1,\ldots, d. 
	\end{array}
\end{equation}
Regarding the $KL_{\alpha}$-transformed sequence of $\{{P}_{n}\}_{n\geqslant0}$, we have: 

\begin{lemma} The polynomial sequence $\{{S}_{n}\}_{n\geqslant0}$  corresponding to the $KL_{\alpha}$-transform of the symmetric $d$-orthogonal Appell polynomials $\{\widehat{H}_{n}\}_{n\geqslant0}$ is a $(2d+2)$-orthogonal polynomial sequence fulfilling 
\begin{eqnarray*}
		{S}_{n+1}\left(\frac{\tau^{2}}{4};\alpha\right) 
	&=&  \left( \frac{\tau^{2}}{4} + \alpha^{2} + (n+1)(n+1+2\alpha) \right)S_{n}\left(\frac{\tau^{2}}{4};\alpha\right) 
		-(d+1)^{-1} \binom{n}{d}  {S}_{n-d}\left(\frac{\tau^{2}}{4};\alpha\right) \\
	&&	+ (2n+2\alpha+1-d)\binom{n}{d+1} 
		{S}_{n-d-1}\left(\frac{\tau^{2}}{4};\alpha\right) \\
	&&
		+ n(n-1) (d+1)^{-1}  \binom{n-d-2}{d}   {S}_{n-2d-2} \left(\frac{\tau^{2}}{4};\alpha\right) \\
	S_{n}\left(\frac{\tau^{2}}{4};\alpha\right) &=& \left((\alpha+1)-\frac{i\tau}{2}\right)_{n}\left((\alpha+1)+\frac{i\tau}{2}\right)_{n}
		\ , \ n=0,1,\ldots, 2d+1.
\end{eqnarray*}
Moreover, 
\begin{equation}\label{Sn Hermite d delta}
	\delta_{i} S_{n}\left(\tfrac{\tau^{2}}{4};\alpha\right)= n S_{n-1}\left(\tfrac{\tau^{2}}{4};\alpha+\tfrac{1}{2}\right)
	\ , \ n\in\mathbb{N}_{0},
\end{equation}
where $\delta_{w}$ represents the operator defined  in \eqref{delta w}. 
\end{lemma}

\begin{proof} The Appell character of a MPS $\{P_{n}\}_{n\geqslant 0}$ implies, in the light of \eqref{Appell MPS}, that 
$$
	\frac{1}{x}\mathcal{A}x P_{n}(x) = n(n-1)x^{2} P_{n-2}(x) + 3 n x P_{n-1}(x) + (1-x)P_{n}(x) 
	\ , \ n\in\mathbb{N}_{0}. 
$$
If, in addition, $\{P_{n}\}_{n\geqslant 0}$ is a $d$-symmetric $d$-MOPS, then $\{P_{n}\}_{n\geqslant 0}$ coincides with $\{\widehat{H}_{n}\}_{n\geqslant0}$ and, in this case, we shall have 
\begin{eqnarray*}
	\left( \frac{1}{x}\mathcal{A}x + 2\alpha \frac{d}{dx}x \right) \widehat{H}_{n}
	&=&  -\widehat{H}_{n+1} + (n+1)(n+1+2\alpha)\widehat{H}_{n} 
		-(d+1)^{-1} \binom{n}{d}  \widehat{H}_{n-d}
		\\
	&&	+ (d+1)^{-1} n \left[(n-1)  \binom{n-2}{d}+ (n+2+2\alpha) \binom{n-1}{d} \right] 
		\widehat{H}_{n-d-1} \\
	&& 	+ n(n-1) (d+1)^{-1}  \binom{n-d-2}{d}   \widehat{H}_{n-2d-2} 
\end{eqnarray*}
{\it i.e.}, 
\begin{eqnarray*}
	\left( \frac{1}{x}\mathcal{A}x + 2\alpha \frac{d}{dx}x \right) \widehat{H}_{n}
	&=&  -\widehat{H}_{n+1} + (n+1)(n+1+2\alpha)\widehat{H}_{n} 
		-(d+1)^{-1} \binom{n}{d}  \widehat{H}_{n-d}
		\\
	&&	+ (2n+2\alpha+1-d)\binom{n}{d+1}  
		\widehat{H}_{n-d-1}  + n(n-1) (d+1)^{-1}  \binom{n-d-2}{d}   \widehat{H}_{n-2d-2} 
		\ , \ n\geqslant 0. 
\end{eqnarray*}

Operating with the $KL_{\alpha}$-transform on both sides of the latter equation and then invoking the relation \eqref{KL alpha differential} with $(m,n)=(1,0)$, we have 
\begin{eqnarray*}
	-\left( \frac{\tau^{2}}{4} + \alpha^{2}\right)S_{n}\left(\frac{\tau^{2}}{4}\right)
	&=&  -{S}_{n+1}\left(\frac{\tau^{2}}{4}\right) + (n+1)(n+1+2\alpha) {S}_{n}\left(\frac{\tau^{2}}{4}\right) 
		-(d+1)^{-1} \binom{n}{d}  {S}_{n-d}\left(\frac{\tau^{2}}{4}\right) \\
	&&	+ (2n+2\alpha+1-d)\binom{n}{d+1} 
		{S}_{n-d-1}\left(\frac{\tau^{2}}{4}\right) \\
	&& 	+ n(n-1) (d+1)^{-1}  \binom{n-d-2}{d}   {S}_{n-2d-2} \left(\frac{\tau^{2}}{4}\right)
	\ , \ n\geqslant0, 
\end{eqnarray*}
which proves the  $(2d+2)$-orthogonality of $\{{S}_{n}\}_{n\geqslant0}$. 

Finally, by virtue of the Appell character of $\{\widehat{H}_{n}\}_{n\geqslant0}$, the relation \eqref{KL alpha + 1/2} implies \eqref{Sn Hermite d delta}. 
\end{proof}

In the light of \eqref{Sn Hermite d delta}, we readily observe the $d$-orthogonality of the latter sequence $\{{S}_{n}\}_{n\geqslant0}$ is preserved under the action of the lowering operator $\delta_{i}$: it is therefore $\delta_{i}$-classical, in the Hahn's sense.

\subsection{Reversed Appell sequences} \label{subsect: Rev Appell}

When $\{P_{n}\}_{n\geqslant0}$   represents an Appell MPS such that $P_{n}(0)\neq0$ for all $n\in\mathbb{N}_{0}$, we may define another MPS $\{R_{n}\}_{n\geqslant0}$ through 
$$
	R_{n}(x) = \frac{1}{\lambda_{n}} x^n P_{n}\left(\frac{1}{x}\right) \ , \ n\in\mathbb{N}_{0},
$$
with $\lambda_{n}=P_{n}(0)$. We will refer to such MPS $\{R_{n}\}_{n\geqslant0}$ as {\it reversed Appell}. 

The Appell character of $\{P_{n}\}_{n\geqslant0}$ 
induces the differential relation for the reversed polynomials $\{R_{n}\}_{n\geqslant0}$   
\begin{equation}\label{diff rel Reversed Appell}
	 \frac{d}{dx} x R_{n+1}(x) = (n+2) R_{n+1}(x) - (n+1) \frac{\lambda_{n}}{\lambda_{n+1}} R_{n}(x) 
	 \ ,  \ n\in\mathbb{N}_{0}.
\end{equation}

Since the work of Toscano \cite{Toscano1956p2} the Laguerre polynomials are known to be the unique OPS such that the reversed polynomial sequence is an Appell. Later on, Cheikh and Douak  \cite[Th.1.2]{BCheikhDouak2001p2} broadened the result to the $d$-orthogonal case: they showed that the unique $d$-MOPS, say $\{R_{n}\}_{n\geqslant0}$, whose reversed sequence  $\{P_{n}\}_{n\geqslant0}$  is of Appell,  is, up to a linear change of variable, given by 
$$
	R_{n}(x):=R_{n}(x;\alpha_{1},\ldots,\alpha_{d})
	=\frac{1}{\lambda_{n}} \ 
	\mathop{{}_{1}F_{d}}\left( {-n \atop \alpha_{1}+1,\ldots,\alpha_{d}+1 } ; x\right)
	\ , \ n\geqslant 0, 
$$
with $-\alpha_{j}\notin \mathbb{N} $ and 
\begin{equation}\label{lambda n Laguerre}
	\lambda_{n}=(-1)^{n}  \left(\prod_{\sigma=1}^{d} (\alpha_{\sigma}+1)_{n}\right)^{-1}
	\ , \ n\geqslant 0. 
\end{equation}
Therefore, the corresponding $KL_{\alpha} $-transform $\{S_{n}\}_{n\geqslant0}$ is given by 
\begin{equation}\label{KL alpha Rev Appell}
	S_{n}(x; d,\alpha,\alpha_{1},\ldots , \alpha_{d})
	=  \frac{(-1)^{n}}{\lambda_{n}} \left(\prod_{\sigma=1}^{d} (\alpha_{\sigma}+1)_{n}\right)
	\mathop{{}_{3}F_{d}}\left( {-n , \alpha+1-\frac{i\tau}{2},\alpha+1+\frac{i\tau}{2}\atop \alpha_{1}+1,\ldots,\alpha_{d}+1 } ; 1\right)
	\ , \ n\geqslant 0. 
\end{equation}

A simple inspection shows that, in particular, the case where $d=2$, gives rise to the (monic) {\it Continuous Dual Hahn polynomials}, which form actually a MOPS. 

The recurrence coefficients of the $d$-MOPS 
\begin{equation}
	\begin{array}{lcl}
	\ds R_{n+2}\left(x\right)  
		&=& \ds\left( x -  {\beta}_{n+1} \right)R_{n+1}\left(x\right) 
		-  \sum_{\nu=0}^{d-1} {\gamma}_{n-\nu+1}^{d-1-\nu}R_{n-\nu}\left(x\right)
	\end{array}
\end{equation}
are known - see \cite{BCheikhDouak2001p2}. In particular the case where $d=1$, we recover the classical Laguerre polynomials 
\begin{equation}\label{rec coef Laguerre}
\left\{	\begin{array}{ccl}
	\beta_{n}(1;\alpha_{1}) 	&=& 	2n+\alpha_{1} +1\ , \ n\geqslant 0, 	\\
	\gamma^{0}_{n+1}(1;\alpha_{1})		&=&	(n+1)(n+\alpha_{1}+1) 
		\ , \ n\geqslant 0, \\
	\gamma^{d-1-\nu}_{n-\nu+1} (1;\alpha_{1}) &=& 0  \ , \ n\geqslant 0, 
\end{array}\right. 
\end{equation}
while the case where $d=2$ corresponds to the sequence characterized in \cite{BCheikhDouak2000} associated to the Bateman's function, whose recurrence coefficients are: 
\begin{equation}\label{rec coef CDH}
	\left\{\begin{array}{l}
	 \beta_{n} := \beta_{n}(\alpha_{1}+1,\alpha_{2}+1) = 3n^2 + (2\alpha_{1}+2\alpha_{2}+3)n + (\alpha_{1}+1)	
			(\alpha_{2}+1) \ , \ n\geqslant 0,  \\
	 \gamma_{n}^1 :=\gamma_{n}^1(\alpha_{1},\alpha_{2}) = n(3n+\alpha_{1}+\alpha_{2})
			(n+\alpha_{1})	
			(n+\alpha_{2}) \ , \ \ n\geqslant 0, \\
	\gamma_{n}^0 :=\gamma_{n}^0(\alpha_{1},\alpha_{2}) = n(n+1)(n+\alpha_{1}+1)(n+\alpha_{1})	
			(n+\alpha_{2}+1)(n+\alpha_{2}) \ , \ \ n\geqslant 0. 
	\end{array}
	\right. 
\end{equation}

\begin{lemma} If  $\{R_{n}\}_{n\geqslant0}$ is a reversed Appell $d$-MOPS, then corresponding $KL_{\alpha}$-transformed sequence  $\{S_{n}\}_{n\geqslant0}$  given by \eqref{KL alpha Rev Appell} is a $\widetilde{d}$-MOPS, with 
$$\widetilde{d}
 	= \left\{\begin{array}{lcl}
	2 	&,& d=1\\
	1 	&,& d=2\\
	d 	&,&    d= 3,4,5,\ldots \ 
	\end{array}\right. , 
$$
fulfilling the recurrence relation 
\begin{equation}\label{Laguerre Sn d-ortho}
\begin{array}{lcl}
	\ds S_{n+2}\left(\tfrac{\tau^2}{4}\right)  
		&=& \ds\left( \frac{\tau^2}{4} -  \widetilde{\beta}_{n+1} \right)S_{n+1}\left(\tfrac{\tau^2}{4}\right) 
		-  \sum_{\nu=0}^{d-1} \widetilde{\gamma}_{n-\nu+1}^{d-1-\nu}S_{n-\nu}\left(\tfrac{\tau^2}{4}\right)
	\end{array}
\end{equation}
where 
\begin{eqnarray*}
		\widetilde{\beta}_{n+1} 
			&=&\beta_{n+1} -  (n+2+\alpha)^2  \\
		 \widetilde{\gamma}_{n+1}^{d-1-\nu}
		 	&=&  \gamma_{n+1}^{d-1}   
			-  (n+1)  ( 2n+3+2\alpha)\prod_{\sigma=1}^{d} \left(n +1+ \alpha_{\sigma} \right) 
			\\
		 \widetilde{\gamma}_{n}^{d-1-\nu}
		 	&=&  \gamma_{n}^{d-2} 
			- n(n+1)  \prod_{\sigma=1}^{d} \left(n+1 
			+ \alpha_{\sigma} \right) \left(n + \alpha_{\sigma} \right)  \\
		 \widetilde{\gamma}_{n-\nu+1}^{d-1-\nu}
		 	&=& \gamma_{n-\nu+1}^{d-1-\nu} \ , \ \nu=2,3,\ldots , d-1
\end{eqnarray*}
with $-\alpha_{j}\notin \mathbb{N} $.
Moreover  $\{\widehat{S}_{n}(\cdot;a,a_{1},\ldots,a_{d})\}_{n\geqslant0}$, with 
$\widehat{S}_{n}(\tau^{2}/4;a; a_{1},\ldots,a_{d}) := S_{n} (\tau^{2}/4; a-1, a+a_{1}-1,\ldots, a+a_{d}-1)$ fulfills 
\begin{equation}\label{delta Sn Laguerre}
	\delta_{i} \widehat{S}_{n+1}(\tfrac{\tau^2}{4};a,a_{1},\ldots,a_{d})
	= n \ \widehat{S}_{n}(\tfrac{\tau^2}{4};a+\tfrac{1}{2},a_{1}+\tfrac{1}{2},\ldots,a_{d}+\tfrac{1}{2})
	\ , \ n\in \mathbb{N}_{0},
\end{equation}
where $\delta_{w}$ represents the operator defined  in \eqref{delta w}. 
\end{lemma}

\begin{proof} By virtue of \eqref{diff rel Reversed Appell}, it follows  
\begin{equation*}
	\begin{array}{@{}ll}
	 \ds\left( \frac{1}{x} \mathcal{A} x + 2\alpha \frac{d}{dx}x \right) R_{n+1}(x) 
	 &= \ds 
	 - xR_{n+1}(x)+ (n+2)(n+2+2\alpha) R_{n+1}(x) \\
	&\ds - (n+1) \frac{\lambda_{n}}{\lambda_{n+1}} ( 2n+3+2\alpha) R_{n}(x) 
		+ n(n+1) \frac{\lambda_{n-1}}{\lambda_{n+1}} R_{n-1}(x)
	\end{array}
\end{equation*}
under the convention $R_{-n}(x)=0$ for $ n\in\mathbb{N}$.
Thus, on the grounds of relation \eqref{KL alpha differential} with  $(m,n)=(1,0)$ on Lemma \ref{lem:KL alpha differential}, the action of the linear operator $KL_{\alpha}$ on both sides of the latter equality  leads to 
$$
		\begin{array}{@{}l}
	 \ds - \left(\frac{\tau^2}{4}+ \alpha^2 \right) KL_{\alpha}[R_{n+1}(x)](\tau) \\
	 = \ds  KL_{\alpha}[- xR_{n+1}(x)](\tau)
	+ (n+2)(n+2+2\alpha) KL_{\alpha}[R_{n+1}(x)](\tau) \\
	\ds - (n+1) \frac{\lambda_{n}}{\lambda_{n+1}} ( 2n+3+2\alpha)KL_{\alpha}[ R_{n}(x) ](\tau)
		+ n(n+1) \frac{\lambda_{n-1}}{\lambda_{n+1}} KL_{\alpha}[R_{n-1}(x)](\tau)
	\end{array}
$$
By virtue of the $d$-orthogonality of $\{R_{n}\}_{n\geqslant0}$ and the  injectivity of the $KL_{\alpha}$-transform, it follows that 
$$
		\begin{array}{@{}l}
	 \ds -\left(\frac{\tau^2}{4}+ \alpha^2 \right) S_{n+1}(\tfrac{\tau^2}{4}) \\
	 = \ds  - \left( S_{n+2}(\tfrac{\tau^2}{4})+ \beta_{n+1} S_{n+1}(\tfrac{\tau^2}{4})
	+  \sum_{\nu=0}^{d-1} \gamma_{n-\nu+1}^{d-1-\nu} S_{n-\nu}(\tfrac{\tau^2}{4})\right)\\
	\ds + (n+2)(n+2+2\alpha) S_{n+1}(\tfrac{\tau^2}{4}) 
	\ds - (n+1) \frac{\lambda_{n}}{\lambda_{n+1}} ( 2n+3+2\alpha)S_{n}(\tfrac{\tau^2}{4})
		+ n(n+1) \frac{\lambda_{n-1}}{\lambda_{n+1}} S_{n-1}(\tfrac{\tau^2}{4})
	\end{array}
$$
which, on account of the expression of $\lambda_{n}$ given in \eqref{lambda n Laguerre}, corresponds to \eqref{Laguerre Sn d-ortho}.

When $d=1$ $\{R_{n}\}_{n\geqslant 0}$ is a MOPS whose recurrence coefficients are given in \eqref{rec coef Laguerre} which, according to \eqref{rec coef CDH} compels  $\{S_{n}\}_{n\geqslant 0}$ to fulfill 
\begin{align} \
	\begin{array}{@{}l}
		S_{n+2}(z;\alpha) 	= (z- (-2 \alpha +\alpha _1-(\alpha +n)^2))S_{n+1}(z;\alpha)
				\vspace{0.2cm}\\
				\qquad 2 (n+1) (\alpha +n+1) \left(\alpha _1+n+1\right)S_{n}(z;\alpha) 
				+n (n+1) \left(\alpha _1+n\right) \left(\alpha _1+n+1\right)S_{n-1}(z;\alpha) \ , 
	\end{array}
\end{align}
ensuring the fact that $\{S_{n}\}_{n\geqslant0}$ is a $2$-orthogonal sequence. 

In turn, the case where $d=2$, the 2-MOPS $\{R_{n}\}_{n\geqslant0}$ fulfills the third order recursive relation whose recurrence coefficients are those given in \eqref{rec coef CDH}. In this case, $\{S_{n}\}_{n\geqslant0}$ is an ($1$-)orthogonal sequence fulfilling 
\begin{equation}\begin{array}{@{}l@{}c@{}l}
	S_{n+2}(z;\alpha) 	&=& \Big(z- (-\alpha  (\alpha +4)+2 n^2+(5-2 \alpha ) n+\alpha _2 (2
   n+3)+\alpha _1 \left(\alpha _2+2 n+3\right)+3)\Big)S_{n+1}(z;\alpha)\\
		&&			-(n+1) \left(\alpha _1+n+1\right) \left(\alpha _2+n+1\right)
   \left(-2 \alpha +\alpha _1+\alpha _2+n\right)	S_{n}(z;\alpha) 
		\end{array}
\end{equation}
In this case, $\{S_{n}:=S_{n}(\cdot;2; a-1,a+b-1,a+c-1)\}_{n\geqslant0}$ are precisely the {\it Continuous Dual Hahn} polynomials. 

Whenever $d\geqslant 3$, the $d$-orthogonality of $\{R_{n}\}_{n\geqslant0}$ implies the $d$-orthogonality of  $\{S_{n}\}_{n\geqslant0}$. 

Finally, independently on the values of the parameter $d$, the $d$-MOPS $\{R_{n}\}_{n\geqslant0}$ satisfies the differential relation 
$$
	\frac{d}{dx} R_{n+1}(x;\alpha_{1},\ldots,\alpha_{d})
	= (n+1) R_{n}(x;\alpha_{1}+1,\ldots,\alpha_{d}+1) \ , \ n\geqslant 0. 
$$
Combining this information with  \eqref{KL alpha + 1/2}, we deduce \eqref{delta Sn Laguerre}.
\end{proof}

Regarding the integral representations of the canonical elements of the dual sequence, the analysis shall be split according to the value of the parameter $d$, specially for the cases where $d=1$ or $2$. 

\subsubsection{Case where $d=1$}\label{subsec: Laguerre d=1}

The regular form $u_{0}:=u_{0}(\alpha_{1})$ corresponding to the monic Laguerre polynomials, $\{\widehat{L}_{n}(\cdot;\alpha_{1})\}_{n\geqslant 0}$ (case where $d=1$) admits the well known integral representation 
$$
	<u_{0},f> = \frac{1}{\Gamma(\alpha_{1}+1)} \int_{0}^{+\infty}
		\e^{-x} x^{\alpha_{1}} f(x) dx 
$$
with $\alpha_{1}\neq -n, \ n\in\mathbb{N}$. Within the light of Example \ref{Example: Parseval Exponential}, this integral representation induces an integral representation for the canonical form $s_{0}:=s_{0}(\alpha,\alpha_{1})$ of the corresponding $KL_{\alpha}$-transformed $2$-MOPS $\{\widehat{S}_{n}(\cdot;\alpha,\alpha_{1})\}_{n\geqslant 0}$. Indeed, after the convenient choice of $\beta = \alpha_{1} -\alpha>0$, the equality   \eqref{Parseval Exponential} permits to deduce 
\begin{eqnarray*}
<u_{0}(\alpha_{1}),f> 
	&=& \frac{1}{\Gamma(\alpha_{1}+1)}
	 \int_{0}^{\infty} x^{\alpha_{1}} \e^{-x} f(x) dx \\
	&=& \frac{\sqrt{\e}}{4\pi\Gamma(\alpha_{1}+1)} \int_{0}^{\infty} 
	KL_{\alpha}[f](\tau)  \frac{
	\left|\Gamma\left(\alpha+1+\frac{i\tau}{2}\right)  
	\Gamma\left(\alpha_{1}-\alpha +\frac{i\tau}{2}\right)\right|^{2}}
	{\left|\Gamma\left( i\tau \right)\right|^{2}}  
	W_{-\alpha_{1}+\alpha+\frac{1}{2},\frac{i\tau}{2}}(1)
	 d\tau \ , \\
	 &=& <s_{0}(\alpha,\alpha_{1}),KL_{\alpha}[f](\tau) > \ .
\end{eqnarray*}
The 2-orthogonality of $\{\widehat{S}_{n}(\cdot;\alpha,\alpha_{1})\}_{n\geqslant 0}$ requires the characterization of the second element of the corresponding dual sequence, whose integral representation can be obtained in a similar way, insofar as $u_{1}(\alpha_{1})= u_{0}(\alpha_{1}+1)-u_{0}(\alpha_{1})$, where $u_{1}$ represents the second term of the dual sequence of $\{\widehat{L}_{n}(\cdot;\alpha_{1})\}_{n\geqslant 0}$.

\subsubsection{Case where $d=2$} 

In this case, the 2-MOPS $\{R_{n}:=R_{n}(\cdot;2, \alpha_{1}, \alpha_{2})\}_{n\geqslant0}$ treated in  \cite{BCheikhDouak2000} (see also \cite{AsscheYakubov2000}) is mapped by the $KL_{\alpha}$-transform into the continuous dual Hahn polynomial sequence $\{S_{n}:=S_{n}(\cdot;\alpha, \alpha_{1}, \alpha_{2})\}_{n\geqslant0}$, as it may be observed from \cite[\S1.3]{KoekoekSwarrtow} with the parameter identification $a=\alpha+1, \ b= \alpha_{1}-\alpha$ and $c=\alpha_{2}-\alpha$. It is natural, therefore, that the properties of first ones induce the properties of the other ones. For instance,  the relations exposed in \cite[(3.5)]{BCheikhDouak2000} 
\begin{align*}
& 	R_{n+1}(x;\alpha_{1} ,\alpha_{2} ) = R_{n+1}(x;\alpha_{1} ,\alpha_{2} +1)+(n+1)(n+
			\alpha_{1} +1) R_{n}(x;\alpha_{1} ,\alpha_{2} +1) \\
&  	R_{n+1}(x;\alpha_{1} ,\alpha_{2} ) = R_{n+1}(x;\alpha_{1} +1,\alpha_{2} )+(n+1)(n+
			\alpha_{2} +1) R_{n}(x;\alpha_{1} +1,\alpha_{2} ) \\
&  	x^n = \sum_{k=0}^n \binom{n}{k} \frac{(\alpha_{1} +1)_{n}(\alpha_{2} +1)_{n}}{(\alpha_{1} 
			+1)_{k}(\alpha_{2} +1)_{k}} R_{k}(x;\alpha_{1} ,\alpha_{2} )\\
&  	R_{n}(x;\alpha_{1} ,\alpha_{2} ) = R_{n}(x;\alpha_{1} +1,\alpha_{2} +1) 
		+ (n+1)(2n+3+\alpha_{1} +\alpha_{2} ) R_{n-1}(x;\alpha_{1} +1,\alpha_{2} +1)		\\
	&	 \hspace{2.7cm} + n(n+1)(n+1+\alpha_{1} )(n+\alpha_{2} +1) R_{n-2}(x;\alpha_{1} +1,\alpha_{2} +1)
\end{align*}
ensure the following ones  
\begin{align*}
 	&S_{n+1}(\tfrac{\tau^{2}}{4};\alpha;\alpha_{1} ,\alpha_{2} ) 
			= S_{n+1}(\tfrac{\tau^{2}}{4};\alpha;\alpha_{1} ,\alpha_{2} +1)
			+(n+1)(n+\alpha_{1} +1) S_{n}(\tfrac{\tau^{2}}{4};\alpha;\alpha_{1} ,\alpha_{2} +1) \\
 	&S_{n+1}(\tfrac{\tau^{2}}{4};\alpha;\alpha_{1} ,\alpha_{2} ) 
			= S_{n+1}(\tfrac{\tau^{2}}{4};\alpha;\alpha_{1} +1,\alpha_{2} )
			+(n+1)(n+\alpha_{2} +1) S_{n}(\tfrac{\tau^{2}}{4};\alpha;\alpha_{1} +1,\alpha_{2} ) \\
 	&\left| \left(\alpha+1+i\frac{\tau}{2}\right)_{n}\right|^{2}
			= \sum_{k=0}^n \binom{n}{k} \frac{(\alpha_{1} +1)_{n}(\alpha_{2} +1)_{n}}
			{(\alpha_{1} +1)_{k}(\alpha_{2} +1)_{k}} S_{k}(\tfrac{\tau^{2}}{4};\alpha;\alpha;\alpha_{1} ,\alpha_{2} )\\
 	&S_{n}(\tfrac{\tau^{2}}{4};\alpha;\alpha_{1} ,\alpha_{2} ) = S_{n}(\tfrac{\tau^{2}}{4};\alpha;\alpha_{1} +1,\alpha_{2} +1) 
		+ (n+1)(2n+3+\alpha_{1} +\alpha_{2} ) S_{n-1}(\tfrac{\tau^{2}}{4};\alpha;\alpha_{1} +1,\alpha_{2} +1)		\\
		& \qquad \qquad\qquad\qquad + n(n+1)(n+1+\alpha_{1} )(n+\alpha_{2} +1) S_{n-2}(\tfrac{\tau^{2}}{4};\alpha;\alpha_{1} +1,\alpha_{2} +1) \ , \ n\geqslant 0. 
\end{align*}
Besides, on account of \eqref{prop KL alpha+n}, the relation 
$$
	xR_{n}(x;\alpha_{1} +1,\alpha_{2} +1)=R_{n+1} (x;\alpha_{1} ,\alpha_{2} ) + (n+\alpha_{1} +1)(n+\alpha_{2} +1)R_{n}(x;\alpha_{1} ,\alpha_{2} )
$$
implies 
$$
	\left(\tfrac{\tau^{2}}{4} + (\alpha+1)^{2}\right) 
	S_{n} (\tfrac{\tau^{2}}{4};\alpha+1;\alpha_{1} +1,\alpha_{2} +1)
	= S_{n+1} (\tfrac{\tau^{2}}{4};\alpha;\alpha_{1} ,\alpha_{2} ) 
	+ (n+\alpha_{1} +1)(n+\alpha_{2} +1)S_{n}(\tfrac{\tau^{2}}{4};\alpha;\alpha_{1} ,\alpha_{2} )
$$

Representing by $\{v_{n}:=v_{n}(\alpha_{1}, \alpha_{2})\}_{n\geqslant0}$ the dual sequence of the 2-MOPS $\{R_{n}:=R_{n}(\cdot;2, \alpha_{1}, \alpha_{2})\}_{n\geqslant0}$, we have \cite{BCheikhDouak2000}: 
\begin{eqnarray*}
	&& <v_{0},f> =\frac{2}{\Gamma(1+\alpha_{1})\Gamma(1+\alpha_{2} )} \int_{0}^\infty f(x) 
			x^{(\alpha_{1}+\alpha_{2} )/2} K_{\alpha_{1}-\alpha_{2} }(2\sqrt{x}) dx\\
	&& <v_{1},f> =\frac{-2}{\Gamma(2+\alpha_{1})\Gamma(2+\alpha_{2} )} \int_{0}^\infty f(x) 
			\left( x^{(\alpha_{1}+\alpha_{2} )/2} K_{\alpha_{1}-\alpha_{2} }(2\sqrt{x}) \right)' dx
\end{eqnarray*}
In the light of relation \eqref{rel Parseval Kitau}, upon the choices of $\beta=\frac{\alpha_{1}+\alpha_{2}}{2} -\alpha$ and $i\mu=\alpha_{1}-\alpha_{2}$, we come out with  
\begin{align} \label{rel v0 regular to s0 CDH}
\begin{array}{lcl}		
	\ds <v_{0}(\alpha_{1},\alpha_{2}),f> 
	&=& \ds  \frac{1}{4\pi\Gamma(1+\alpha_{1})\Gamma(1+\alpha_{2} )
		\Gamma(\alpha_{1}+\alpha_{2} -2\alpha ) } \\
	&&	\ds \times \int_{0}^{\infty} KL_{\alpha}[f](\tau)
		 \frac{\left| 
		\Gamma\left(\alpha_{1}-\alpha+\frac{i\tau}{2}\right)
		\Gamma\left(\alpha_{2}-\alpha+\frac{i\tau}{2}\right)  
		\Gamma\left(\alpha+1+\frac{i\tau}{2}\right)\right|^{2}}
		{ \left| \Gamma\left(i\tau\right)\right|^{2}}
		d\tau \\
	&=& \ds <s_{0}(\alpha_{1},\alpha_{2};\alpha), KL_{\alpha}[f](\tau)> 
\end{array}
\end{align}
where $\{s_{n}(\alpha_{1}, \alpha_{2};\alpha)\}_{n\geqslant0}$ represents the dual sequence of the (monic) continuous dual Hahn (orthogonal) polynomial sequence. 

In particular, if we consider $f(x)=x^{n}$ and we take into account \eqref{CentralFact coeffs}, we readily observe the linear relation between the two  sequences of moments. 
The form $v_{0}(\alpha_{1},\alpha_{2})$ is obviously a regular form, but the characterization of the recurrence coefficients of the corresponding monic orthogonal polynomial sequence remains unsolved. 
The $KL_{\alpha}$-operator despite mapping this regular form into the regular form corresponding to the continuous dual Hahn polynomials, it is not a map between orthogonal polynomial sequences, as it will become clear after Theorem \ref{Thm: classification}. Notwithstanding this, another suitable manipulation with these index transforms seems to bring advances to the problem of the characterization of such MOPS. These considerations will be placed in an ensuing work. 

It is worth to mention that, while working with the so-called $J$-matrix method, in \cite[\S3]{IsmailKoelink} the authors showed how the continuous dual Hahn polynomials are involved in the tridiagonalization of an operator connected to the Laguerre polynomials.

\section{The $KL_{\alpha}$-transform of a MOPS}\label{sec: KL MOPS to dMOPS}

In the precedent section were brought into analysis orthogonal polynomial sequences whose $KL_{\alpha}$-transform is a $d$-orthogonal sequence. The examples were the classical sequences of Hermite and Laguerre. So this triggers the quest of finding all the orthogonal sequences that share the same property - Theorem \ref{Thm: classification} brings the answer. 

Prior to this goal, we analyze the connection between the coefficients of a MPS $\{B_{n}\}_{n\geqslant 0}$ and the corresponding $KL_{\alpha}$-transformed sequence, $\{S_{n}(\cdot,\alpha)\}_{n\geqslant 0}$:
\begin{equation}\label{Sn def}
	S_{n}(\tfrac{\tau^{2}}{4};\alpha):=S_{n}(\tfrac{\tau^{2}}{4}) = KL_{\alpha}[B_{n}](\tau) \ , \ n\in\mathbb{N}_{0}, 
\end{equation}
without further assumptions over the two sequences, except eventually, the orthogonality of $\{B_{n}\}_{n\geqslant 0}$. 

In the light of Lemma \ref{lem:KL alpha differential}, when $m=1$ and $n=0$ the relation \eqref{KL alpha differential} ensures 
\begin{equation}\label{tau Sn KL diff Bn eq1}
	-  \left( \frac{\tau^{2}}{4} + \alpha^{2}\right)\ S_{n}(\tau^{2}/4) 
	= KL_{\alpha}\left[\left( \frac{1}{x} \mathcal{A} x + 2\alpha \frac{d}{dx}x \right)  B_{n}(x)\right](\tau) 
	\ , \ n\in\mathbb{N}_{0}.
\end{equation}
No matter the MPS $\{S_{n}\}_{n\geqslant 0}$ considered, it always admits the structural relation 
\begin{equation}\label{struct rel Sn}
	S_{n+1}( \tfrac{\tau^{2}}{4}) 
	 = \left(\tfrac{\tau^{2}}{4}- \zeta_{n} \right) S_{n}( \tfrac{\tau^{2}}{4}) 
	 -  \sum_{\nu=0}^{n-1} a_{n,\nu} S_{\nu}(\tfrac{\tau^{2}}{4})
	 \ , \ n\in\mathbb{N}_{0},
\end{equation}
and therefore, from \eqref{tau Sn KL diff Bn eq1} we obtain 
$$
	-S_{n+1}(\tfrac{\tau^{2}}{4}) - (\zeta_{n} +\alpha^{2}) S_{n}(\tfrac{\tau^{2}}{4})
	- \sum_{\nu=0}^{n-1} a_{n,\nu} S_{\nu}(\tau^{2}/4)
	=KL_{\alpha}\left[\left( \frac{1}{x} \mathcal{A} x + 2\alpha \frac{d}{dx}x \right)  B_{n}(x)\right](\tau) 
	\ , \ n\in\mathbb{N}_{0}, 
$$
which, on account of \eqref{Sn def} and then arguing with the linearity and injectivity of the $KL_{\alpha}$-operator, leads to 
\begin{equation} \label{rec differential for any MOPS}
	-B_{n+1}(x) - (\zeta_{n} +\alpha^{2}) B_{n}(x)
	- \sum_{\nu=0}^{n-1} a_{n,\nu} B_{\nu}(x)
	= \left( \frac{1}{x} \mathcal{A} x + 2\alpha \frac{d}{dx}x \right) B_{n}(x) \ , \ n\in\mathbb{N}_{0}. 
\end{equation}
Consider the expansion of the elements of $\{B_{n}\}_{n\geqslant0}$ 
$$
	B_{n}(x) = x^n + \sum_{\nu=0}^{n-1} b_{n,\nu} x^{\nu}\ , \ n\geqslant 0,
$$
under the convention $b_{n,n}=1$ and $b_{k,m}=0$ whenever $k<m$ or $m<0$. 
Therefore, by equating the coefficients of the powers of $x$ in \eqref{rec differential for any MOPS}, it follows 
\begin{equation}\label{bn zeta and ann 1}
	b_{n+1,\nu}
	= b_{n,\nu-1} -\left((\nu+1+\alpha)^{2}  +\zeta_{n}  \right)b_{n,\nu}
		- \sum_{\mu=\nu}^{n-1} a_{n,\mu} b_{\mu,\nu} \ , \ 
		0\leqslant \nu\leqslant n+1.
\end{equation}

On the other hand, whenever $\{B_{n}\}_{n\geqslant0}$ is  a MOPS, fulfilling the recurrence relation 
\begin{equation}\label{Rec Rel Bn}
	B_{n+2}(x) = (x-\beta_{n+1})B_{n+1}(x) - \gamma_{n+1}B_{n}(x) \ ,  \ n\in\mathbb{N}_{0} ,
\end{equation}
with $B_{0}(x)=1$ and $B_{1}(x)=x-\beta_{0}$, we necessarily have 
$$
	\beta_{n} = b_{n,n-1} - b_{n+1,n}
	\quad , \quad 
	\gamma_{n+1} = b_{n,n-2} - b_{n+1,n-1} -\beta_{n+1} b_{n+1,n} \ , \ n\geqslant 0. 
$$
and 
$$
	b_{n+1,\nu} = b_{n,\nu-1} - \beta_{n} b_{n,\nu}  - \gamma_{n} b_{n-1,\nu} \ .
$$
The insertion of the latter equality in \eqref{bn zeta and ann 1} provides 
$$
	\beta_{n} b_{n,\nu} + \gamma_{n} b_{n-1,\nu}
	= \left((\nu+1+\alpha)^{2}  +\zeta_{n}  \right)b_{n,\nu}
		+ \sum_{\mu=\nu}^{n-1} a_{n,\mu} b_{\mu,\nu}
$$
The particular choice of $\nu=n$ 
\begin{equation}\label{zeta n and beta n}
	\zeta_{n} = \beta_{n} - (\alpha+n+1)^{2} \ , \ n\geqslant 0, 
\end{equation}
whence 
$$
	 \gamma_{n} b_{n-1,\nu}
	= -(n-\nu ) (2 \alpha +\nu +n+2)b_{n,\nu}
		+ \sum_{\mu=\nu}^{n-1} a_{n,\mu} b_{\mu,\nu}
$$
For $\nu=n-1$, we obtain 
\begin{equation}\label{an n-1 in terms of bn}
	a_{n,n-1} = \gamma_{n} + (2n+2\alpha-1) b_{n,n-1}
			=  \gamma_{n} -  (2n+2\alpha-1) \sum_{\mu=0}^{n-1} \beta_{\nu}
			\ , \ n\geqslant 1, 
\end{equation}
while the remaining ones can be recursively computed via
%
\begin{equation}\label{an nu in terms of bn}
	a_{n,n-\nu} = -(2n+2\alpha-1) b_{n,n-1} b_{n-1,n-\nu}
	+ \nu (2 \alpha +2n-\nu+2) b_{n,n-\nu}
	-  \sum_{\mu=2}^{\nu-1} a_{n,n-\mu} b_{n-\mu,n-\nu}
	\ , \ 2\leqslant \nu\leqslant n . 
\end{equation}

When further assumptions, like differential properties or finite type relations, are made over the MOPS $\{B_{n}\}_{n\geqslant 0}$ or the MPS $\{S_{n}\}_{n\geqslant 0}$, then we can afford to characterize both sequences. 


The quest regarding the class of all orthogonal polynomial sequences whose $KL_{\alpha}$-transform is a $d$-orthogonal polynomial sequence is clarified in the next result. The precedent examples of Hermite or Laguerre polynomials already ensure this collection is not empty. 
Hence, we will assume from this point forth that $\{S_{n}\}_{n\geqslant 0}$ is $d$-orthogonal, which amounts to the same as requiring the aforementioned coefficients $a_{n,n-\nu}=0$ for $\nu=0,1,\ldots,d-1$ whenever $n\geqslant d$. The next result aims to characterize all the possible solutions of this problem.

\begin{theorem}\label{Thm: classification} Let $\{B_{n}\}_{n\geqslant 0}$ be a MOPS with respect to $u_{0}$. If the corresponding $KL_{\alpha}$-transformed sequence $\{S_{n}\}_{n\geqslant 0}$ is a $d$-MOPS, then $d$ must be an even integer greater or equal than 2 and $\{B_{n}\}_{n\geqslant 0}$ is necessarily a semiclassical sequence of class $s\in\{\max(0,d/2-2),d/2-1,d/2\}$. The corresponding regular form $u_{0}$ fulfills 
$$
	D(\phi u_{0}) + \psi u_{0}=0
$$
with $(\phi,\psi)$ representing a pair of admissible polynomials given by one of the following expressions: 
\begin{enumerate}
\item[{\bf a)}] $\phi(x)=x^{2}$ and $\psi(x)= x (N \rho(x) - (3+2\alpha))$ with $\rho(0)=0$, $\langle u_{0}, \rho(x)\rangle \neq N^{-1} (2+2\alpha)$ and $\alpha\neq-\frac{n+3}{2}$, $n\in\mathbb{N}_{0} $ ($u_{0}$ is of class $s=\frac{d}{2}$); 

\item[{\bf b)}] $\phi(x)=x$ and $\psi(x)= N \rho(x) - (2+2\alpha)$ with $\langle u_{0}, \rho(x)\rangle \neq N^{-1} (2+2\alpha)$ and $\alpha\neq-\frac{n}{2}-1$, for $n\geqslant 1$ (the class of $u_{0}$ is $s=\frac{d}{2}-1$);

\item[{\bf c)}] $\phi(x)=1$ and $\psi(x)= N \theta_{0}\rho(x)$ with $N\rho(0)=1+2\alpha$  (the class of $u_{0}$ is $s=\frac{d}{2}-2$ as long as $d\geqslant 4$);
\end{enumerate}
where $\rho(x)$ is a monic polynomial such that $\deg \rho (x)=\frac{d}{2}$ and  $N\neq0$ is a normalization constant. Moreover, the MOPS $\{B_{n}\}_{n\geqslant 0}$ fulfills 
\begin{equation}\label{diff rel for Bn in thm}
	\begin{array}{l}
	\ds x^{2} B_{n}''(x) +x \Big(N \rho(x) - (3+2\alpha)\Big) B_{n}'(x) \\
	\ds - \Big\{ N\rho(x) \Big(N \rho(x) - (2+2\alpha) \Big) - N x \rho'(x)-x+(1+2\alpha)  \Big\} B_{n} 
		=- \sum_{\nu=n-1}^{n+d} \rho_{n,\nu}^d B_{\nu}(x)
	\end{array}
\end{equation}
where 
$$\rho_{n,\mu}^d =\left\{ 
	\begin{array}{ccl}
	\ds \frac{\langle u_{0},B_{n}^2\rangle}{\langle u_{0},B_{\mu}^2\rangle}
		\alpha_{n+1}^{n+d-\mu} &\text{ if } & n+1\leqslant \mu\leqslant n+d, 
			\ \text{with } \ n\geqslant 0, \\
	\ds \zeta_{n}-\alpha^2 &\text{ if } & \mu=n , \ \text{with } \ n\geqslant 0, \\
	\ds \gamma_{1} &\text{ if } & \mu=n-1 \ \text{with } \ n\geqslant 1\ .
\end{array}\right.$$
\end{theorem}

\begin{proof} According to \eqref{Ttranspose}, by duality, we readily deduce the adjoint or transpose operator - 
${}^{t}\!\!\left( \frac{1}{x} \mathcal{A} x + 2\alpha \frac{d}{dx}x \right)= \left(   \mathcal{A}  - 2\alpha x\frac{d}{dx} \right)$ - in order to deduce from  \eqref{rec differential for any MOPS} the corresponding relation fulfilled by the dual sequence  $\{u_{n}\}_{n\geqslant 0}$ of  the MPS $\{B_{n}\}_{n\geqslant 0}$.  
Indeed, \eqref{rec differential for any MOPS} implies 
\begin{eqnarray*}
	<  \left(   \mathcal{A}  - 2\alpha x\frac{d}{dx} \right)u_{n} \ , \  B_{\nu} >
	&=& - \delta_{n,\nu+1} -  (\zeta_{\nu} -\alpha^{2}) \delta_{n,\nu} 
	- \sum_{\mu=0}^{\nu-1} a_{\nu,\mu} \delta_{n,\mu}
\end{eqnarray*}
which, on the grounds of  \eqref{u in terms of un}, brings 
$$
	 \left(   \mathcal{A}  - 2\alpha x\frac{d}{dx} \right)(u_{n})
	 = - u_{n-1} - (\zeta_{n} -\alpha^{2}) u_{n} 
	 -  \sum_{\nu\geqslant n+1} a_{\nu,n} u_{\nu} \ , \ n\geqslant 0. 
$$
The orthogonality of $\{B_{n}\}_{n\geqslant0}$ not only implies the second order recurrence relation \eqref{Rec Rel Bn}, but also corresponds to write $u_{n}=(<u_{0},B_{n}^{2}>)^{-1}B_{n}u_{0}$, and therefore  
\begin{equation}\label{proof thm rel without d ortho}
	 \left(   \mathcal{A}  - 2\alpha x\frac{d}{dx} \right)(B_{n}u_{0})
	 = -\gamma_{n} B_{n-1}u_{0} - (\zeta_{n} -\alpha^{2}) B_{n} u_{0} 
	 -  \sum_{\nu\geqslant n+1} a_{\nu,n} \frac{<u_{0},B_{n}^{2}>}{<u_{0},B_{\nu}^{2}>} B_{\nu} u_{0}
\end{equation}

Since  $\{S_{n}\}_{n\geqslant0}$ is a $d$-orthogonal, it fulfills  \eqref{struct rel Sn} where 
\begin{align*}
	& a_{n,\nu} = \alpha_{\nu+1}^{d+\nu-n}
			\quad \text{ with } \quad 
			\alpha_{\nu+1}^0\neq0 \ , \quad n-d\leqslant \nu\leqslant n-1,
			\ n\geqslant d, 
	\\
	& a_{n,\nu}=0 \ , \ 0\leqslant \nu \leqslant n-d-1 \ , \ n\geqslant d \\
	& a_{n,\nu} = \alpha_{\nu+1}^{d+\nu-n} \ , \ 0\leqslant \nu\leqslant n-1 \ , \ 1\leqslant n\leqslant d-1.
\end{align*}
The recurrence coefficients $(\zeta_{n}, \alpha_{n+1}^{d-1},\ldots \alpha_{n+1}^{0})_{n\geqslant 0}$ can of course be computed provided we know the necessary data of the original sequence $\{B_{n}\}_{n\geqslant0}$, via the relations \eqref{zeta n and beta n}-\eqref{an nu in terms of bn}. 

The $d$-orthogonality $\{S_{n}\}_{n\geqslant0}$, permits to transform the right hand side of \eqref{proof thm rel without d ortho} into a finite sum. Thus, by collecting its left-hand side in terms of $B_{n}$, $B_{n}'$ and $B_{n}''$, relation \eqref{proof thm rel without d ortho} becomes 
\begin{equation}\label{CondPn OrthoRel when Sn d ortho f2}
	B_{n}\Big(    \left( \mathcal{A} -2\alpha D x +2\alpha \right) u_{0} \Big)
	+ B_{n}' \Big(  2(x^{2} u_{0})'- (2\alpha +3) x \, u_{0} \Big) 
	+ B_{n}'' x^{2} u_{0} 
	= - A_{n}^d(x) u_{0}\ , \ n\geqslant 0, 
\end{equation}
with  
\begin{equation}\label{D polys}
	A_{n}^d(x) = \sum_{\mu=0}^{d-1} \frac{\langle u_{0},B_{n}^2\rangle}{\langle u_{0},B_{n+\mu+1}^2\rangle}
		\alpha_{n+1}^{d-1-\mu} B_{n+\mu+1} +\left( \zeta_{n} -\alpha^2\right)B_{n}+ \gamma_{n} B_{n-1}
\end{equation}
under the convention $\gamma_{0}=0$. The fact $\alpha_{n+1}^0\neq0$ ensures $\deg( A_{n}^d ) = n+d$, $n\geqslant 0$.

The particular choice of $n=0$ in \eqref{CondPn OrthoRel when Sn d ortho f2} brings 
\begin{equation}	\label{eq11 u0 when d Ortho}
	   \left( \mathcal{A} -2\alpha D x +2\alpha \right) u_{0} = -A_{0}^d (x) u_{0} 
\end{equation}
i.e., 
\begin{equation}	\label{eq12 u0 when d Ortho}
	\left(x^2 u_{0}\right)''-(3+2\alpha) \left(x u_{0}\right)' + (A_{0}^d (x)-x+1+2\alpha)u_{0}=0 ,
\end{equation}
while  \eqref{CondPn OrthoRel when Sn d ortho f2} with $n=1$ corresponds to   
\begin{equation}	\label{eq2 u0 when d Ortho}
	2(x^{2} u_{0})'-(3+2\alpha) x u_{0} = \left\{ -A_{1}^d(x) +  A_{0}^d(x) B_{1}(x)  \right\}u_{0}.
\end{equation}
Inserting \eqref{eq11 u0 when d Ortho} and \eqref{eq2 u0 when d Ortho} into \eqref{CondPn OrthoRel when Sn d ortho f2} leads to the differential-recursive relation 
\begin{equation*}
		-A_{0}^d (x) B_{n} u_{0} 
	+\left\{ -A_{1}^d(x) +  A_{0}^d(x) B_{1}(x)  \right\} B_{n}'  u_{0}
	+ B_{n}'' x^{2} u_{0} 
	=- A_{n}^d(x) u_{0} \ , \ n\geqslant 0
\end{equation*}
which, due to the regularity of $u_{0}$, becomes 
\begin{equation}\label{proof: diff rel for Bn}
	x^{2} B_{n}''(x) 
	-\left\{ A_{1}^d(x) -  A_{0}^d(x) B_{1}(x)  \right\} B_{n}'(x)
	=- A_{n}^d(x)+A_{0}^d (x) B_{n}  \ , \ n\geqslant 0.
\end{equation}

Between \eqref{eq12 u0 when d Ortho} and \eqref{eq2 u0 when d Ortho} after a single differentiation
\begin{eqnarray*}
	&& 	2\left(x^2 u_{0}\right)''-2(3+2\alpha) \left(x u_{0}\right)' + 2(A_{0}^d (x)-x+1+2\alpha)u_{0}=0 \\
	&& 	2(x^{2} u_{0})'' - \left(\left\{(3+2\alpha)x -A_{1}^d(x) +  A_{0}^d(x) B_{1}(x)  \right\}u_{0} \right)' =0
\end{eqnarray*}
we deduce 
\begin{equation}\label{eq13 u0 when d Ortho} 
	 \bigg(\left\{-(3+2\alpha)x -A_{1}^d(x) +  A_{0}^d(x) B_{1}(x)  \right\}u_{0} \bigg)' 
	 +2\bigg(A_{0}^d(x)-x+1+2\alpha\bigg)u_{0}=0
\end{equation}
Considering 
\begin{equation}\label{phi1phi2 psi1psi2}
	\begin{array}{lcl}
	 \phi_{1}(x) = x^{2}   &;&   
	\psi_{1}(x)
		=\frac{1}{2}  N_{2}\phi_{2}(x) -(3+2\alpha)x \vspace{0.3cm}
		\\
	\phi_{2}(x) = N_{2}^{-1} \left( A_{1}^d(x) -   A_{0}^d(x) B_{1}(x) + (3+2\alpha)x \right)  
	 &;&  
	\psi_{2}(x)=- 2N_{2}^{-1} \left(A_{0}^d(x) - x + 1+ 2\alpha \right)
\end{array}\end{equation}
we then have from \eqref{eq2 u0 when d Ortho} and \eqref{eq13 u0 when d Ortho} that $u_{0}$ fulfills 
\begin{equation}\label{dOPS eq u0 1 and 2}
	D(\phi_{j} u_{0}) + \psi_{j} u_{0}=0 \ , \ j=1,2, 
\end{equation}
In particular and on account of the regularity of $u_{0}$, the two latter equations provide the condition 
\begin{equation}\label{Cond dOrtho}
	\phi_{2}(x) \Big( \frac{1}{2}  N_{2}\phi_{2}(x) -(1+2\alpha)x \Big) = x^2 (\psi_{2}(x) + \phi_{2}'(x)). 
\end{equation}

By taking $x\to0$ in \eqref{Cond dOrtho}, we readily observe that $\phi_{2}(0)=0$, thus there is a monic polynomial $\widetilde{\phi}_{2}$ such that $\phi_{2}(x)=x \widetilde{\phi}_{2}(x)$ and the condition \eqref{Cond dOrtho} can be becomes 
\begin{equation}\label{Cond dOrtho 2}
	\psi_{2}(x)
	= \widetilde{\phi}_{2}(x) \Big( \frac{1}{2}  N_{2}\widetilde{\phi}_{2}(x) -(1+2\alpha) \Big) 
	- (x \widetilde{\phi}_{2}(x))'  \ .
\end{equation}
The regularity of $u_{0}$ implies that $\deg\widetilde{\phi}_2(x) \geqslant 1$, otherwise the condition $\deg\psi_{2}\geqslant 1$ would be contradicted  (whence, by recalling its definition in \eqref{phi1phi2 psi1psi2},  $\deg \psi_{2}=d$).  Thus, an inspection about the degrees of the polynomials on both sides of \eqref{Cond dOrtho 2} leads to 
$$
	2\deg (\widetilde{\phi}_{2}) = \deg \Big(\psi_{2}(x) +(x \widetilde{\phi}_{2}(x))'\Big)
$$
which implies $1\leqslant\deg\widetilde{\phi}_{2} \leqslant d-1$ and, therefore, 
$
	\deg (\widetilde{\phi}_{2}) = \frac{d}{2} =( \deg \psi_{1} )-1 \ .
$

Consequently, representing by $\rho(x):=\widetilde{\phi}_{2}(x)$ and $N:=\frac{N_{2}}{2}$, the regular form $u_{0}$ fulfills the two following equations 
\begin{align*}
	& x\Big((x u_{0})' + (N\rho(x) - (2+2\alpha)) u_{0} \Big)=0 \\
	& \rho(x) \Big( (x u_{0})' + (N\rho(x) - (2+2\alpha)) u_{0} \Big) =0
\end{align*}
In case $\rho(0) \neq 0 $, then we necessarily have 
\begin{equation}\label{thmproof: diff eq for u0 case b}
	(x u_{0})' + (N\rho(x) - (2+2\alpha)) u_{0} = 0
\end{equation}
otherwise, if $\rho(0)=0$ then $u_{0}$ would fulfill 
\begin{equation}\label{thmproof: diff eq for u0 case a}
	(x^{2} u_{0})' + (N {\rho(x)} - (3+2\alpha)) x\, u_{0} = 0
\end{equation}
which could be simplified into \eqref{thmproof: diff eq for u0 case b} if 
$\langle u_{0} , (N{\rho(x)} - (2+2\alpha))\rangle = 0 $. 

This ensures that only three possible situations are admissible: 
\begin{description}
\item[{\rm Case a) }] The assumptions $\rho(0)=0$ and $\langle u_{0} , (N{\rho(x)} - (2+2\alpha))\rangle \neq 0 $ with $\alpha\neq-\frac{n+3}{2}$, $n\geqslant0$,  ensure the irreducibility of \eqref{thmproof: diff eq for u0 case a}. In this case $u_{0}$ is a semiclassical form of class $s=\frac{d}{2}$. 

\item[{\rm Case b)}] Whenever $\langle u_{0} , (N {\rho(x)} - (2+2\alpha))\rangle = 0 $ (no matter whether $\rho(0)$ equals zero or not),  but as long as 
$$\Big| N \rho(0) - (1+2\alpha) \Big| + \Big|<u_{0},  (\theta_{0}\rho)(x)>\Big| \neq0 , $$
then the class of the semiclassical form $u_{0}$ is determined by the equation \eqref{thmproof: diff eq for u0 case b}. 

\item[{\rm Case c)}] In case  $ N \rho(0) -(1+2\alpha) = 0 = <u_{0},  (\theta_{0}\rho)(x)> 
=\langle u_{0} , (N {\rho(x)} - (2+2\alpha))\rangle = 0$ and as long as $d\geqslant 4$, the semiclassical form $u_{0}$ fulfills 
$$
	D(u_{0}) + N \,  (\theta_{0}\rho)(x) u_{0} =0 \ ,
$$
whose class is $\frac{d}{2}-2$. We notice that necessarily there is a set of coefficients, say $\{c_{\nu}: \nu=1,\ldots, \frac{d}{2}-2\}$, such that $ (\theta_{0}\rho)(x) = \sum_{\nu=1}^{d/2 -2} c_{\nu} B_{\nu}(x)$ with $Nc_{1} = {\gamma_{1}}^{-1}$ (inasmuch as $<u_{0},  N x (\theta_{0}\rho)(x)>= 1 $). 

\end{description}

Finally,  on account of \eqref{phi1phi2 psi1psi2} together with \eqref{Cond dOrtho 2}, the relation  \eqref{proof: diff rel for Bn} can be rewritten as in \eqref{diff rel for Bn in thm}. 
\end{proof}

\begin{remark}
 The only MOPSs whose $KL_{\alpha}$-transformed sequence is a $d$-MOPS are semiclassical sequences whose class ranges between $\frac{d}{2}-2$ and $\frac{d}{2}$, where $d$ must be an even number. The classical solutions to this problem (case where $s=0$) can only be the Hermite or Laguerre polynomial sequences (up to a linear change of variable). 
\end{remark}

In case a) of Theorem \ref{Thm: classification}, the class of the semiclassical form $u_{0}$ is $s=\frac{d}{2}$. We can actually observe that the form $ v = \lambda^{-1}xu_{0}$ (with $\lambda=\beta_{0}\neq0$) fulfills 
\begin{equation}\label{eq for v=xu}
	D(x v) + (N\rho(x)-(3+2\alpha)) v =0 \ . 
\end{equation}
The form $v$ is regular (ergo, semiclassical of class $\widehat{s}:=s(v)=\frac{d}{2}-1$) if and only if $B_{n}(0)\neq0$ for all $n\geqslant 1$ and then the corresponding MOPS $\{V_{n}\}_{n\geqslant0}$ is related to the MOPS $\{B_{n}\}_{n\geqslant0}$ through the finite-type relations 
\begin{eqnarray*}
	&& 	x V_{n+1}(x)= B_{n+1}(x) - \frac{B_{n+1}(0)}{B_{n}(0)} B_{n}(x) \ , \ n\geqslant 0 , \\
	&&  B_{n+1}(x) = V_{n+1}(x) -  \frac{B_{n}(0)}{B_{n+1}(0)} \gamma_{n+1} V_{n}(x) \ , \ n\geqslant 0. 
\end{eqnarray*}

This means that $u_{0}= \lambda x^{-1} v + \delta$, where $\delta$ represents the linear functional $\langle\delta,f\rangle =f(0)$ for any $f\in\mathcal{P}$. 

Conversely, if the form $v$ fulfilling \eqref{eq for v=xu} is regular (ergo, semiclassical), then the form $u_{0}= \lambda x^{-1} v + \delta$ is regular (and semiclassical) as long as $\lambda$ do not belong to a discrete set of singular values. Precisely, we recall:

\begin{proposition}\cite{MaroniPerHung 1990}\label{Prop Maroni Semiclassical} Let $w$ be a regular form, semiclassical of class $s$ fulfilling $D(\phi_{w} w) + \psi_{w} w =0$ and let $\{W_{n}\}_{n\geqslant0}$ be the corresponding MOPS. The form $u= \lambda x^{-1} w + \delta$ is regular and semiclassical of class $s+1$ if and only if $\lambda\neq\lambda_{n}$, $n\geqslant 0$, where 
$$ 
	\lambda_{0}=0 \ ; \ \lambda_{n+1} = - \frac{W_{n+1}(0)}{\langle w, \theta_{0} W_{n+1}(x)\rangle}
	\ , \ n\geqslant 0, 
$$ 
and such that $ \langle v, 
				\theta_{0}^{2}\phi_{w} + \theta_{0} \psi_{w}\rangle \lambda
				+ \phi_{w}'(0) + \psi_{w}(0) \neq0 $. 
Moreover, the recurrence coefficients corresponding to the MOPS $\{P_{n}\}_{n\geqslant 0}$ with respect to  $u$ are related to those of $\{W_{n}\}_{n\geqslant0}$ through 
$$
	\beta_{n}^{P}  = \beta_{n}^{W} + a_{n-1} - a_{n} 
	\qquad , \qquad 
	\gamma_{n}^{P} = - a_{n}(a_{n} - \beta_{n}^{W}) \ , \ n\geqslant 0,
$$
where $ a_{-1}=0$ and 
$$
	a_{n} = - \frac{W_{n+1}(0) + \lambda \langle w, \theta_{0} W_{n+1}(x)\rangle}
			{W_{n}(0) + \lambda \langle w, \theta_{0} W_{n}(x)\rangle}\ , \ n\geqslant 0, 
$$
and we have 
$$
	P_{n+1}(x) = W_{n+1}(x) + a_{n} W_{n}(x) \ , \ n\geqslant 0\ .
$$
\end{proposition}

Consequently, the form $u_{0}:=u_{0}(\alpha)$, with $\alpha\neq-\frac{n+3}{2}$ is regular and semiclassical of class $s=d/2$ when the form $v$ is a semiclassical form of class $d/2-1$, providing that $\lambda\neq 0, - \frac{V_{n+1}(0)}{\langle v, \theta_{0} V_{n+1}(x)\rangle}	\ , \ n\geqslant 0 $.

In order to have a more precise idea of the possible cases that may arise, we shall focus on the case $d=2$. In \S\ref{subsect: Rev Appell}, we have already clarify that the $KL_{\alpha}$-transform of a Laguerre polynomial sequence is a $2$-orthogonal sequence. We will next show that, apart from this sequence, another one may arise: a semiclassical sequence that is a linear combination of two consequent elements of a Laguerre sequence.

\subsection{The case where $d=2$} \label{subsec2: case d=2}

In the light of Theorem \eqref{Thm: classification}, the $KL_{\alpha}$-transform of a MOPS $\{B_{n}\}_{n\geqslant0}$ is a 2-orthogonal sequence, then necessarily $\{B_{n}\}_{n\geqslant0}$ is a semiclassical sequence of class $1$ or $0$, as $u_{0}$ fulfills 
$$
	D(x^{2} u_{0}) + x(N (x+\rho(0)) - (3+2\alpha)) u_{0},
$$
where $\rho$ is a monic polynomial of degree 1. Depending on the structure of the polynomial $\rho$, two situations may arise.

{\bf Case a)} $\rho(0)=0$ and $\langle u_{0}, N x - (2+2\alpha) \rangle \neq 0$ with $\alpha \neqÊ- \frac{n+3}{2}$, $n\geqslant 0$. The semiclassical character of a linear functional remains invariant by shifting, therefore the form $\widetilde{u}_{0}=h_{N}u_{0}$ fulfills 
$$
	D(x^{2} \widetilde{u}_{0}) + x(x -(3+2\alpha))\widetilde{u}_{0}=0 \ .
$$
In this case, the form $v$ defined through $\lambda v=x \widetilde{u}_{0}$, i.e., $\widetilde{u}_{0} = \lambda \, x^{-1} v + \delta $ fulfills 
$$
	D(x v ) + (x - (3+2\alpha)) v =0\ . 
$$
Thus, $v$ is regular insofar as it is a Laguerre form of parameter $2\alpha+2$: $v=\mathfrak{L}(2\alpha+2)$. Therefore \cite[Proposition 4.2.]{MarTounsi2011}  $\widetilde{u}_{0}$ is a semiclassical form of class 1 if and only if $\lambda\notin\{0\ , \  \left(c_{n}(\alpha)\right)^{-1};\  n\geqslant 0\}$ where 
$$
	c_{n}(\alpha) =\left\{ \begin{array}{lcl}
		\frac{\Gamma (n+2 \alpha +4)-\Gamma (2 \alpha +3) \Gamma (n+2)}{(2 \alpha +2)
   \Gamma (n+2 \alpha +4)} &,& \alpha\neq-1\\
   		\sum\limits_{\nu=1}^{n+1} \frac{1}{\nu} &,& \alpha=-1, 
	\end{array}\right.\ , \ n\geqslant 0.
$$
The recurrence coefficients of the polynomial sequence $\{\widetilde{B}_{n}(x):=N^{n}B_{n}(N^{-1} x)\}_{n\geqslant 0}$ orthogonal with respect to $\widetilde{u}_{0}$ are then given by 
$$\begin{array}{l}
	\widetilde{\beta}_{0} = \lambda ; \quad 
	\widetilde{\beta}_{n+1} = 2n+2\alpha +5+ a_{n}-a_{n+1} \ , \  n\geqslant 0,  \\ 
	\widetilde{\gamma}_{n+1} = - a_{n}(a_{n}-2n-2\alpha-3) \ ,  \  n\geqslant 0,
\end{array}$$
where 
$$
	a_{n}=(n+2\alpha+3) \frac{1-\lambda c_{n}(\alpha)}{1-\lambda c_{n-1}(\alpha)} 
	 =\frac{
	 \lambda (n+1)! +(2 \alpha -\lambda +2)   (2 \alpha +3)_{n+1}}
   	{\lambda \, n! 	+(2 \alpha -\lambda +2) (2 \alpha +3)_{n}} \ , \ n\geqslant 0. 
$$
The MOPS  $\{\widetilde{B}_{n}\}_{n\geqslant 0}$ is related to the monic Laguerre polynomial sequence $\{\widehat{L}_{n}(\cdot;2\alpha+2)\}_{n\geqslant 0}$ through 
\begin{eqnarray*}
&&	\widetilde{B}_{n+1}(x) = \widehat{L}_{n+1}(x;2\alpha+2) + a_{n}  \widehat{L}_{n}(x;2\alpha+2)
	\ , \ n\geqslant 0,  \\
&&	x \widehat{L}_{n+1}(x;2\alpha+2) = \widetilde{B}_{n+1}(x) - (a_{n}-(2n-2\alpha-3))\widetilde{B}_{n}(x)
	\ , \ n\geqslant 0.  
\end{eqnarray*}

The integral representation for the form $\widetilde{u}_{0}=\lambda x^{-1}v+\delta$ can be deduced from the one of $v=\mathfrak{L}(2\alpha+2)$, for ${\rm Re}(\alpha)>-1$, which we recall 
$$
	\langle v,f\rangle = \frac{1}{\Gamma(2\alpha+3)} \int_{0}^{+\infty} x^{2\alpha+2} \e^{-x} f(x) dx 
	\ , \ \  \forall f\in\mathcal{P},
$$
and therefore 
$$
	\langle \widetilde{u}_{0},f\rangle = \frac{\lambda}{\Gamma(2\alpha+3)} 
		\int_{0}^{+\infty} x^{2\alpha+1} \e^{-x} f(x) dx + \left(1-\frac{\lambda}{2\alpha+2}\right)f(0)
			\ , \ \  \forall f\in\mathcal{P}.
$$

The 2-MOPS $\{\widetilde{S}_{n}\}_{n\geqslant 0}$ that is the $KL_{\alpha}$-transform of $\{\widetilde{B}_{n}\}_{n\geqslant 0}$ is given by 
$$
	\widetilde{S}_{n}(\tfrac{\tau^{2}}{4})
	= \widehat{S}^{L}_{n+1}(\tfrac{\tau^{2}}{4} ) + a_{n}  \widehat{S}^{L}_{n}(\tfrac{\tau^{2}}{4} )
	\ , \ n\geqslant 0, 
$$
where $\{\widehat{S}^{L}_{n}\}_{n\geqslant 0}$ represents the 2-MOPS that is the  $KL_{\alpha}$-transform  of the monic Laguerre polynomials of parameter $(2\alpha+2)$. Following the considerations made on \S\ref{subsect: Rev Appell}, we have 
$$
	 \widehat{S}^{L}_{n}(\tfrac{\tau^{2}}{4} )= KL_{\alpha}[\widehat{L}_{n}(x;2\alpha+2)](\tau)
	=  (-1)^{n}   (2\alpha+3)_{n} 
	\mathop{{}_{3}F_{1}}\left( {-n , \alpha+1-\frac{i\tau}{2},\alpha+1+\frac{i\tau}{2}\atop 2\alpha+3 } ; 1\right)
	\ , \ n\geqslant 0. 
$$

{\bf Case b) } When $\langle u_{0}, N \rho(x) - (2+2\alpha) \rangle = 0$ and $\deg\rho=1$, then 
$N\rho(x)=N \, B_{1}(x) + (2+2\alpha)$ and $u_{0}$ fulfills 
$$
	D(x u_{0}) + N B_{1}(x) u_{0} =0 \ .
$$
Looking at the moment equation, it follows 
$$
	N(u_{0})_{n+1} = (n+N\beta_{0}) (u_{0})_{n} \ , \ n\geqslant 0.
$$
In particular, when $n=1$, it follows 
$
	N\gamma_{1} = \beta_{0} \ .
$

Similarly to the precedent case, it is more conveninent to deal with the form $\widetilde{u}_{0}=h_{N} u_{0}$ (instead of $u_{0}$), which fulfills 
$$
	D(x \widetilde{u}_{0} ) + (x-\alpha_{1}-1) \widetilde{u}_{0}= 0
$$
after setting $\beta_{0}/N = \alpha_{1}+1$. Thus, $\widetilde{u}_{0}$ is a Laguerre form of parameter $\alpha_{1}$. Considerations regarding the $KL_{\alpha}$-transformed sequence were made in \S\ref{subsect: Rev Appell}.

%

\subsection{The cases where $d\geqslant 4$ and some final remarks} \label{subsec2: case d=4}

 In the light of Theorem \ref{Thm: classification}, a MOPS whose $KL_{\alpha}$-transformed sequence is $4$-orthogonal is necessarily a semiclassical sequence of class $s$ equal to $0,1$ or $ 2$. The classical polynomial sequence (case where $s=0$) leads to the Hermite polynomials, whereas the generalized Hermite polynomials arise as a particular example of the case b, where necessarily $s=1$. 
 The semiclassical sequences of class $2$ comes out within the line of Proposition \ref{Prop Maroni Semiclassical}. We defer the study of all these solutions to a future work. The analysis of the semiclassical sequences that arise for higher (even) values of the parameter $d$ is as well left to a later task, specially regarding the technical and hard computations involved.

Finally, the example of treated on  \S\ref{subsect: Rev Appell} of an orthogonal sequence that is the $KL_{\alpha}$-transformed of a $2$-orthogonal sequence, raises the question of seeking all the other $d$-orthogonal sequences that are mapped into orthogonal ones.

\bibliographystyle{amsplain}

\end{document}